\documentclass[12pt]{article}
\usepackage{latexsym,amsmath}
\usepackage{amssymb}
\usepackage{verbatim}
\usepackage{amsthm}
\usepackage{tikz}
\usepackage{url}
\usepackage[margin=25mm]{geometry}

\newtheorem{thm}{Theorem}[section]
\newtheorem{lemma}[thm]{Lemma}

\newtheorem{cor}[thm]{Corollary}

\newtheorem{claim}[thm]{Claim}

\theoremstyle{definition}
\newtheorem{defn}[thm]{Definition}

\theoremstyle{definition}
\newtheorem{ex}[thm]{Example}

\theoremstyle{definition}

\theoremstyle{definition}
\newtheorem{question}[thm]{Question}

\theoremstyle{definition}

\theoremstyle{definition}
\newtheorem{alg}[thm]{Algorithm}

\theoremstyle{remark}
\newtheorem{remark}[thm]{Remark}

\newcommand{\ord}{\mathrm{ord}}

\begin{document}

\title{Restricted Stirling and Lah number matrices and their inverses}
\author{John Engbers\thanks{Department of Mathematics, Statistics and Computer Science, Marquette University,
Milwaukee WI. Research supported by the Simons Foundation grant 524418.}\and David Galvin\thanks{Department of Mathematics, University of Notre Dame, Notre Dame IN. Research supported by the Simons Foundation grant 360240 and by the National Security Agency grant NSA H98230-13-1-0248.} \and Cliff Smyth\thanks{Department of Mathematics and Statistics, University of North Carolina at Greensboro, Greensboro NC. Research supported by the Simons Foundation grant 360468.}}
\date{\today}

\maketitle

\begin{abstract}
Given $R \subseteq \mathbb{N}$ let ${n \brace k}_R$, ${n \brack k}_R$, and $L(n,k)_R$ count the number of ways of partitioning the set $[n]:=\{1,2, \ldots,n\}$ into $k$ non-empty subsets, cycles and lists, respectively, with each block having cardinality in $R$. We refer to these as the $R$-restricted Stirling numbers of the second and first kind and the $R$-restricted Lah numbers, respectively. Note that the classical Stirling numbers of the second kind and first kind, and Lah numbers are ${n \brace k} = {n \brace k}_{\mathbb{N}}$, ${n \brack k} = {n \brack k}_{\mathbb{N}} $ and $L(n,k) = L(n,k)_{\mathbb{N}}$, respectively.

It is well-known that the infinite matrices $[{n \brace k}]_{n,k \geq 1}$, $[{n \brack k}]_{n,k \geq 1}$ and $[L(n,k)]_{n,k \geq 1}$ have inverses $[(-1)^{n-k}{n \brack k}]_{n,k \geq 1}$, $[(-1)^{n-k} {n \brace k}]_{n,k \geq 1}$ and $[(-1)^{n-k} L(n,k)]_{n,k \geq 1}$ respectively. The inverse matrices $[{n \brace k}_R]^{-1}_{n,k \geq 1}$, $[{n \brack k}_R]^{-1}_{n,k \geq 1}$ and $[L(n,k)_R]^{-1}_{n,k \geq 1}$ exist and have integer entries if and only if $1 \in R$.  We express each entry of each of these matrices as the difference between the cardinalities of two explicitly defined families of labeled forests. In particular the entries of $[ {n \brace k}_{[r]}]^{-1}_{n,k \geq 1}$ have combinatorial interpretations, affirmatively answering a question of Choi, Long, Ng and Smith from 2006.

If we have $1,2 \in R$ and if for all $n \in R$ with $n$ odd and $n \geq 3$, we have $n \pm 1 \in R$, we  additionally show that each entry of $[{n \brace k}_R]^{-1}_{n,k \geq 1}$, $[{n \brack k}_R]^{-1}_{n,k \geq 1}$ and $[L(n,k)_R]^{-1}_{n,k \geq 1}$ is up to an explicit sign the cardinality of a single explicitly defined family of labeled forests. With $R$ as before we also do the same for restriction sets of the form $R(d) = \{ d(r-1) + 1 : r \in R\}$ for all $d \geq 1$. Our results also provide combinatorial interpretations of the $k$th Whitney numbers of the first and second kinds of $\Pi_n^{1,d}$, the poset of partitions of $[n]$ that have each part size congruent to $1$ mod $d$.
\end{abstract}

\section{Introduction} \label{sec-intro-AAA}

For all integers $n,k \geq 1$, let ${n \brace k}$, ${n \brack k}$, and $L(n,k)$ be the classical Stirling numbers of the second and first kinds, and Lah numbers, respectively.  These numbers are defined as follows: ${n \brace k}$ is the number of partitions of $[n] := \{1,2, \ldots, n\}$ into $k$ non-empty subsets, ${n \brack k}$ is the number of partitions of $[n]$ into $k$ non-empty cyclically ordered sets, i.e.\!\! cycles, and $L(n,k)$ is the number of partitions of $[n]$ into $k$ non-empty linearly ordered sets, i.e.\!\! lists.  All of our partitions will be unordered unless we specify otherwise. Let $S_2 := [ {n \brace k} ]_{n,k \geq 1}$, $S_1 := [ {n \brack k}]_{n,k \geq 1}$, and $L := [ L(n,k) ]_{n,k \geq 1}$ be infinite matrices with rows and columns indexed by the natural numbers $\mathbb{N} := \{1, 2, \ldots\}$. In this notation $n$ is the row index and $k$ is the column index.  It is well-known that $S_2^{-1} = [ (-1)^{n-k} {n \brack k} ]_{n,k \geq 1}$, $S_1^{-1} = [ (-1)^{n-k} {n \brace k} ]_{n,k \geq 1}$ and $L^{-1} = [ (-1)^{n-k} L(n,k) ]_{n,k \geq 1}$.  In particular, each entry of each inverse matrix has, up to sign, a combinatorial interpretation.

We consider the following generalizations of Stirling and Lah numbers.  
\begin{defn} \label{def:restricted-Stirling-numbers}
For $R \subseteq \mathbb{N}$ the \emph{$R$-restricted Stirling number of the second kind}, ${n \brace k}_R$, is the number of partitions of $[n]$ into $k$ non-empty subsets such that the cardinality of each subset is restricted to lie in $R$. Analogously, the \emph{$R$-restricted Stirling numbers of the first kind} ${n \brack k}_R$ and \emph{$R$-restricted Lah numbers} $L(n,k)_R$ are the numbers of partitions of $[n]$ into $k$ cycles and lists, respectively, with cardinalities restricted to lie in $R$.
\end{defn}
Note that we recover the classical Stirling numbers of both kinds and the Lah numbers by taking $R$ to be ${\mathbb N}$ (e.g. ${n \brace k}_{\mathbb N} = {n \brace k}$ etc.).

Various instances of restricted numbers have appeared in the literature. Comtet \cite[page 222]{Comtet} introduced {\em $r$-associated Stirling numbers of the second kind}, ${n \brace k}_R$ with ${R}=\{r, r+1, r+2, \ldots\}$, and obtained recurrence relations and generating functions for them. Belbachir and Bousbaa \cite{BelbachirBousbaa} studied \emph{$r$-associated Lah numbers}, $L(n,k)_R$ also with $R = \{r, r+1, r+2, \ldots\}$.  Choi and Smith \cite{ChoiSmith} considered {\em $r$-restricted Stirling numbers of the second kind}, ${n \brace k}_R$ with $R=[r]$.

We extend the classical results on the inverses of Stirling and Lah number matrices to find combinatorial formulas for the inverses of $R$-restricted Stirling and Lah number matrices whenever the inverses exist, i.e., whenever $1 \in R$.

\begin{defn} \label{def:inverse-numbers}
Denote by ${n \brace k}^{-1}_R$ (${n \brack k}^{-1}_R$, $L(n,k)^{-1}_R$) the entry in the $n$th row and $k$th column of the matrix $[ {n \brace k}_R]^{-1}_{n,k \geq 1}$ ($[ {n \brack k}_R ]^{-1}_{n,k \geq 1}$, $[L(n,k)_R]^{-1}_{n,k \geq 1}$, respectively), when the inverse matrix exists.  We refer to ${n \brace k}^{-1}_R$ as the \emph{inverse $R$-restricted Stirling number of the second kind}, ${n \brack k}^{-1}_R$ as the \emph{inverse $R$-restricted Stirling number of the first kind}, and $L(n,k)^{-1}_R$ as the \emph{inverse $R$-restricted Lah number}.
\end{defn}
Our first result (Theorem \ref{thm:forest-difference}) is that for all $R \subseteq \mathbb{N}$ with $1 \in R$, ${ n \brace k}^{-1}_R$, ${ n \brack k}^{-1}_R$, and $L(n,k)^{-1}_R$ can each be expressed as the difference between the cardinalities of two explicitly defined sets of forests. 

If $R$ has more structure, we can say more.
\begin{defn} \label{def:no-exposed-odds}
Say that $R \subseteq \mathbb{N}$ {\em has no exposed odds} if it has the following properties: 
\begin{enumerate}
\item if $1 \in R$ then $2 \in R$ and
\item if $n$ is odd, $n \geq 3$, and $n \in R$ then $n-1, n+1 \in R$. 
\end{enumerate}
\end{defn}
For $d\geq 1$ and $R \subseteq \mathbb{N}$ set $R(d) := \{ d(n-1)+1: n \in R\}$.
We view $R(d)$ as the set $R$ ``stretched'' along the arithmetic progression $\{1, d+1, 2d+1, \ldots\}$. Our main set of results (Theorems \ref{thm:single-forest} and \ref{thm:single-forest-stretched}) is that, for all $R \subseteq \mathbb{N}$ with $1 \in R$ and with no exposed odds, and for all $d \geq 1$, each of
${ n \brace k}^{-1}_{R}$, ${n \brack k}^{-1}_{R}$, $L(n,k)^{-1}_{R}$, ${ n \brace k}^{-1}_{R(d)}$, ${n \brack k}^{-1}_{R(d)}$, and $L(n,k)^{-1}_{R(d)}$ can be expressed, up to an explicit sign, as the cardinality of a
single explicitly defined set of forests.

In \cite{ChoiLongNgSmith} Choi, Long, Ng and Smith note that ${n \brace k}^{-1}_{[2]}$ is a Bessel number \cite[A100861]{sloan} and has many combinatorial interpretations. For example, $(-1)^{n-k} {n \brace k}^{-1}_{[2]}$ counts the number of size $n-k$ matchings of the complete graph $K_{2n-1-k}$ \cite{ChoiSmith2}.  They asked if ${n \brace k}^{-1}_{[r]}$ has a combinatorial interpretation for $r > 2$, and observed that an anomalous sign behavior in ${n \brace k}^{-1}_{[3]}$
presents an obstacle to any such interpretations. 

But in fact our results provide  such combinatorial interpretations, and these are particularly nice whenever $r$ is even; see Corollary \ref{cor:special-cases} (Part 1) below.

We give below, in Corollary \ref{cor:special-cases}, some illustrative special cases of the results in our paper. We also give some applications to calculating the Whitney numbers of a certain subposet of the partition lattice (Theorem \ref{thm:Whitney}).

Recall that a plane tree is a rooted tree in which the set of children of each vertex of the tree are given a linear ordering from left to right.  If the leaves of a tree are labeled with integers we extend that labeling to other vertices $v$ by setting $\ell(v)$ to be the maximum of the labels of the leaves descended from $v$.  Let ${\cal H}(n,k)$ be the set of forests consisting of an unordered collection of $k$ plane rooted trees: (i) with $n$ leaves in total (an isolated root is considered a leaf) (ii) with all non-leaves having at least two children and (iii) with the leaves labeled with the integers $1$ through $n$ in such a way that $\ell(v)$ increases from left to right across each set of siblings.

\begin{cor} \label{cor:special-cases} 

The following are special cases of Definition \ref{def:R-good}, Claim \ref{claim:3inR}, and Theorems \ref{thm:single-forest} and \ref{thm:single-forest-stretched}.

\begin{enumerate}

\item Let $r \geq 1$. The number ${ n \brace k}^{-1}_{\{1,2, \ldots, 2r\}}$ is $(-1)^{n-k}$ times the number of forests in ${\cal H}(n,k)$ in which each vertex $v$ has $0$, $2$, or $2r$ children unless $v$ is the left-most child of a vertex with two children, in which case it has $0$ or $2r$ children.

\item Let $r \geq 1$ and $d \geq 2$. If $n \equiv k \pmod{d}$, then ${ n \brace k}^{-1}_{\{1,d+1, 2d+1 , \ldots, 1 + (2r-1)d\}}$ is $(-1)^{(n-k)/d}$ times the number of forests in ${\cal H}(n,k)$ in which each vertex $v$ has $0$, $d+1$ or $1 + (2r-1)d$ children unless $v$ is the left-most child of a vertex with $d+1$ children, in which case it has $0$ or $1 + (2r-1)d$ children. If $n \not \equiv k \pmod{d}$, then the number is $0$.

\item Let $d \geq 1$.  If $n \equiv k \pmod{d}$, then ${n \brace k}^{-1}_{\{1, d+1, 2d+1, \ldots\}}$ is $(-1)^{(n-k)/d}$ times the number of forests in ${\cal H}(n,k)$ in which each vertex has $0$ or $d+1$ children and in which left-most children are always leaves.  If $n \not \equiv k \pmod{d}$, then the number is $0$.

\end{enumerate}

\end{cor}

Suppose $P$ is a finite ranked poset with unique minimal element $0$.  For all $k \geq 0$, the {\em $k$th Whitney number of the second kind}, $W_k(P)$, is the number of elements of $P$ of rank $k$ and the {\em $k$th Whitney number of the first kind}, $w_k(P)$, is given by $w_k(P) = \sum_{x} \mu(0,x)$ where $\mu$ is the M\"obius function of $P$ and $x$ ranges over the elements of $P$ of rank $k$.  The theory of subposets of the set partition lattice $\Pi_n$ consisting of partitions with restricted part sizes has received considerable attention in the literature, see for instance \cite{Calderbank,StanleyExp,Sylvester,Wachs}.  Our results give combinatorial interpretations of the Whitney numbers of the ranked poset $\Pi_n^{1,d}$ consisting of all partitions of $[n]$ that have each part size congruent to $1$ mod $d$.

\begin{thm} \label{thm:Whitney}
For all $n,d \geq 1$ and $k \geq 0$ we have

\[W_k(\Pi_n^{1,d}) = {n \brace n - k d}_{\{1, d+1, 2d+1, \ldots\}}\] and
\[w_k(\Pi_n^{1,d}) = {n \brace n - k d}^{-1}_{\{1, d+1, 2d+1, \ldots \}}\]

In particular, $w_k(\Pi_n^{1,d})$ is $(-1)^{k}$ times the number of forests in ${\cal H}(n,n-kd)$ in which each vertex has $0$ or $d+1$ children and in which left-most children are always leaves.
\end{thm}

Our paper is organized as follows.  We provide definitions related to our combinatorial interpretations in Section \ref{sec:notation} and then
state our main results in Section \ref{sec:results}.  In Section \ref{sec:preliminaries}, we state some preliminary lemmas.  We give proofs of our main results in Section \ref{sec:proofs}. In Section \ref{sec:conclusion} we note some connections to known number sequences and indicate some directions for future research.

\section{Notation} \label{sec:notation}

As is evident from Corollary \ref{cor:special-cases} and Theorem \ref{thm:Whitney}, trees and forests figure heavily in our results. Our trees will all be {\em rooted}, i.e.\! they will come with a distinguished root vertex. Our forests will also all be {\em rooted}, i.e.\! they will consist of unordered collections of rooted trees. Let $F$ be a rooted forest and let $v$ and $w$ be vertices of $F$.  If $v$ lies on the path from $w$ to a root, then $v$ is an {\em ancestor} of $w$ and $w$ is a {\em descendant} of $v$.  If, in addition, $v$ and $w$ are neighbors, we say $v$ is the {\em parent} of $w$ and $w$ is a {\em child} of $v$.  We say $v$ and $w$ are {\em siblings} if they have the same parent.  The {\em degree} or {\em down-degree} of $v$, denoted $d_F(v)$, is the number of children of $v$ in $F$.  We say $v$ is a {\em leaf} of $F$ if $d_F(v) = 0$.  Note that by our definition, isolated roots are also leaves.

Our forests will either have ordered children or unordered children.
A forest with {\em unordered children} is just a graph made up of rooted trees with no ordering on sets of siblings. A forest has {\em ordered children} if the set of children of each non-leaf vertex $v$ is given a specific linear order from {\em left-most} to {\em right-most}.  Although a rooted tree with ordered children is usually called a plane tree we avoid this terminology as we do not consider plane forests, i.e. linearly ordered collections of plane trees. The components of our forests will always be unordered.

If $T$ is a tree, a {\em leaf-labeling} of $T$ is an injective map $\ell$ from the leaves of $T$ to $\mathbb{N}$.  A leaf-labeling of a tree with $n$ leaves is {\em proper} if it has range $[n]$. We will work with two extensions of a leaf-labeling to non-leaf vertices.
\begin{defn} \label{def:leaf-labels}
Given a leaf-labeling $\ell$ of the leaves of a tree $T$, the labeling $\ell_{\max}$ on the vertices of $T$ is defined by setting $\ell_{\max}(v)$ to be the maximum of the labels of the leaves descended from $v$. The labeling $\ell_{\min}$ is defined by setting $\ell_{\min}(v)$ to be the minimum of the labels of the leaves descended from $v$.  
\end{defn}
Note that any two children of a vertex have distinct labels with respect to the $\ell_{\max}$ (or $\ell_{\min}$) labeling.

A {\em phylogenetic tree (forest)} is a rooted tree (forest) with unordered children such that no vertex has degree $1$, together with a proper leaf-labeling. For $1 \leq k \leq n$, we define ${\cal T}(n)$ to be the family of phylogenetic trees on $n$ leaves and ${\cal F}(n,k)$ to be the family of phylogenetic forests with $n$ leaves and $k$ unordered components. Also, let ${\cal T}^{\rm even}(n)$ denote the subset of trees in ${\cal T}(n)$ that have an even number of edges, and let ${\cal T}^{\rm odd}(n)$ be the complementary set of trees with an odd number of edges.

\begin{defn} \label{def:special-trees}
Let $G$ be a phylogenetic tree or forest. If each complete set of siblings (full set of children of a non-leaf vertex of $G$) is assigned a linear ordering, we say that $G$ is a {\em linearly ordered phylogenetic tree (forest)}. We say $G$ is {\em increasingly ordered} if $G$ is linearly ordered and if additionally for each complete set of siblings, the $\ell_{\max}$ label of the siblings increases from left to right.  We say $G$ is {\em min-first ordered} if $G$ is linearly ordered and if additionally for each complete set of siblings, the left-most sibling has the smallest $\ell_{\min}$ label amongst all the siblings. 
\end{defn}
Let ${\cal T}^{\rm i.o.}(n)$, ${\cal T}^{\rm m.o.}(n)$, and ${\cal T}^{\rm l.o.}(n)$ be the families of increasingly ordered, min-first ordered, and linearly ordered phylogenetic trees on $n$ leaves, respectively.  For all $1 \leq k \leq n$ we define ${\cal F}^{\rm i.o.}(n,k)$ (${\cal F}^{\rm m.o.}(n,k)$, ${\cal F}^{\rm l.o.}(n,k)$) to be the family of increasingly (min-first, linearly) ordered phylogenetic forests on $n$ leaves with $k$ unordered components.

If $R \subseteq \mathbb{N}$ and ${\cal C}$ is any class of trees or forests, we write ${\cal C}_R$ for the subclass of objects in ${\cal C}$ which have all non-zero down-degrees lying in $R$.  For example, ${\cal T}^{\rm i.o.}_R(n)$ is the set of all increasingly ordered phylogenetic trees with $n$ leaves and all non-zero down-degrees lying in $R$.

For $d \geq 1$ let $s_d:\mathbb{N} \to \mathbb{N}$ be defined by $s_d(n):=d(n-1)+1$.  As we defined in the introduction, let $R(d) = s_d(R) = \{ d(n-1)+1: n \in R\}$. Note that $s_1$ is the identity and $R(1) = R$.  

\begin{defn} \label{def:internal-sequence}  Let $R \subseteq \mathbb{N}$ and let $d \geq 1$.  If $G$ is a phylogenetic forest with all down-degrees in $R(d)$ let $(v_i)_{i=1}^m$ be some arbitrary but fixed ordered list of the non-leaf vertices of $G$. For each $i$ let $n_i$ be the unique integer such that $d(v_i)=s_d(n_i)$. We refer to $(n_i)_{i=1}^m$ as the \emph{internal sequence} of $G$.  We say that $G$ is {\em even} if $\sum_{i=1}^m n_i$ is even and \emph{odd} otherwise. 
\end{defn}

Note that if $d=1$ then $n_i = d(v_i)$ and $\sum_{i=1}^m n_i$ is just the number of edges of $G$. We define ${\cal T}_{R(d)}^{\rm i.o., even}(n)$ (${\cal T}_{R(d)}^{\rm i.o., odd}(n)$) to be the sets of even (odd) increasingly ordered trees on $n$ leaves with down-degrees in $R(d)$ and define the analogous notations for the other possible subclasses of even and odd ordered trees and forests.  For example ${\cal F}_{R(d)}^{\rm m.o., odd}(n,k)$ is the set of odd min-first ordered phylogenetic forests with down-degrees in $R(d)$ and with $n$ leaves and $k$ components.  If $d=1$ then, since $R(1) = R$, we will write this as ${\cal F}_{R}^{\rm m.o., odd}(n,k)$.

\section{Results} \label{sec:results}

In this section we state our main results. Using a formula for combinatorial Lagrange inversion we obtain the following combinatorial interpretation for each inverse $R$-restricted number (with $1 \in R$) as the difference in cardinality between two sets of forests.
\begin{thm} \label{thm:forest-difference}

Let $R \subseteq {\mathbb N}$.  Then ${n \brace k}^{-1}_{R}$, ${n \brack k}^{-1}_{R}$, and $L(n,k)^{-1}_{R}$ exist if and only if $1 \in R$. For all $R$ with $1 \in R$ and all $n,k \geq 1$ we have
$$
{n \brace k}^{-1}_{R} = (-1)^{n-k} \left(\left|{\mathcal F}^{\rm i.o., even}_R(n,k)\right|-\left|{\mathcal F}^{\rm i.o., odd}_{R}(n,k)\right|\right),
$$
$$
{n \brack k}^{-1}_{R} = (-1)^{n-k} \left( \left|{\mathcal F}^{\rm m.o., even}_R(n,k)\right|-\left|{\mathcal F}^{\rm m.o., odd}_{R}(n,k)\right|\right),
$$
$$
L(n,k)^{-1}_{R} = (-1)^{n-k} \left(\left|{\mathcal F}^{\rm l.o., even}_R(n,k)\right|-\left|{\mathcal F}^{\rm l.o., odd}_{R}(n,k)\right|\right).
$$
\end{thm}

Recall (Definition \ref{def:no-exposed-odds}) that $R \subseteq \mathbb{N}$ has no exposed odds if (i) $2 \in R$ whenever $1 \in R$, and (ii) $n-1, n+1 \in R$ whenever $n \in R$, $n \geq 3$, and $n$ is odd. Our main result is that for $R$ containing $1$ and with no exposed odds, we can express each inverse entry, up to sign, as the cardinality of a \emph{single set} of forests. We next define the terms needed to describe these sets.

We write $R$ as a disjoint union of its maximal intervals.  Thus if $R$ has no exposed odds it is a union of intervals of the form $[1, \infty)$, $[1,b]$ with $b$ even, $[a,\infty)$ with $a$ even, or $[a,b]$ with $a \leq b$ and $a$ and $b$ even.  Let $a(R)$ be the set of all left endpoints of the intervals in this decomposition of $R$, except $1$, and let $b(R)$ be the set of all right endpoints. Note that if $R = \mathbb{N} = [1, \infty)$ then $a(R)$ and $b(R)$ are empty. Note also that if $[x,x] = \{x\}$ is one of the maximal intervals of $R$, then $x \in a(R)$ and $x \in b(R)$.

\begin{defn}\label{def:rightpath}
Let $v$ be a vertex in a linearly ordered tree or forest $G$.  Then $v$ has \emph{$2$-left-odd ancestry} if $v$ has some ancestor $v_1$ with the following properties:
\begin{itemize}
\item along the path $v_1, \ldots, v_k=v$ from $v_1$ to $v$, for each $1 \leq i < k$ it holds that $d(v_i)=2$, $v_{i+1}$ is a left-most child of $v_i$, and $k$ is even, and
\item $v_1$ is not a left-most child of a vertex $w$ with $d(w)=2$.
\end{itemize}
For $d \geq 1$, we say $v$ has \emph{$s_d(2)$-left-odd ancestry} if $v$ has some ancestor $v_1$ such that
\begin{itemize}
\item along the path $v_1, \ldots, v_k=v$ from $v_1$ to $v$, for each $1 \leq i < k$ it holds that $d(v_i)=s_d(2)$, $v_{i+1}$ is a left-most child of $v_i$, and $k$ is even, and
\item $v_1$ is not a left-most child of a vertex $w$ with $d(w)=s_d(2)$.
\end{itemize}
\end{defn}

In Figure \ref{fig-Rgood}(a), only vertex $w_2$ has $2$-left-odd ancestry. In Figures \ref{fig-Rgood}(b) and \ref{fig-Rgood}(c), only vertices $w_2$ and $w_4$ have $2$-left-odd ancestry. 

\begin{figure}[ht!]
\begin{center}
\begin{tikzpicture}[scale=.9,every node/.style={scale=0.95}]

\node at (-0.5,5) {(a)};

\node (a1) at (1,5) [circle,draw,label=180:$w_1$] {};
\node (a2) at (1.5,4) [circle,draw,label=270:$1$] {};
\node (a3) at (0.5,4) [circle,draw,label=180:$w_2$] {};
\node (a4) at (1,3) [circle,draw,label=270:$2$] {};
\node (a5) at (0,3) [circle,draw,label=180:$w_3$] {};
\node (a6) at (0.75,2) [circle,draw,label=270:$3$] {};
\node (a7) at (0.25,2) [circle,draw,label=270:$4$] {};
\node (a8) at (-0.25,2) [circle,draw,label=270:$5$] {};
\node (a9) at (-0.75,2) [circle,draw,label=270:$6$] {};

\foreach \from/\to in {a1/a2,a1/a3,a3/a4,a3/a5,a5/a6,a5/a7,a5/a8,a5/a9} \draw (\from) -- (\to);

\node at (4.5,5) {(b)};
\node (b1) at (6,5) [circle,draw,label=180:$w_1$] {};
\node (b2) at (6.5,4) [circle,draw,label=270:$1$] {};
\node (b3) at (5.5,4) [circle,draw,label=180:$w_2$] {};
\node (b4) at (6,3) [circle,draw,label=270:$2$] {};
\node (b5) at (5,3) [circle,draw,label=180:$w_3$] {};
\node (b6) at (5.5,2) [circle,draw,label=270:$3$] {};
\node (b7) at (4.5,2) [circle,draw,label=180:$w_4$] {};
\node (b8) at (5.25,1) [circle,draw,label=270:$4$] {};
\node (b9) at (4.75,1) [circle,draw,label=270:$5$] {};
\node (b10) at (4.25,1) [circle,draw,label=270:$6$] {};
\node (b11) at (3.75,1) [circle,draw,label=270:$7$] {};

\foreach \from/\to in {b1/b2,b1/b3,b3/b4,b3/b5,b5/b6,b5/b7,b7/b8,b7/b9,b7/b10,b7/b11} \draw (\from) -- (\to);

\node at (9.5,5) {(c)};
\node (c1) at (11,5) [circle,draw,label=180:$w_1$] {};
\node (c2) at (11.5,4) [circle,draw,label=270:$1$] {};
\node (c3) at (10.5,4) [circle,draw,label=180:$w_2$] {};
\node (c4) at (11,3) [circle,draw,label=270:$2$] {};
\node (c5) at (10,3) [circle,draw,label=180:$w_3$] {};
\node (c6) at (10.5,2) [circle,draw,label=270:$3$] {};
\node (c7) at (9.5,2) [circle,draw,label=180:$w_4$] {};
\node (c8) at (10,1) [circle,draw,label=270:$4$] {};
\node (c9) at (9,1) [circle,draw,label=180:$w_5$] {};
\node (c10) at (9.75,0) [circle,draw,label=270:$5$] {};
\node (c11) at (9.25,0) [circle,draw,label=270:$6$] {};
\node (c12) at (8.75,0) [circle,draw,label=270:$7$] {};
\node (c13) at (8.25,0) [circle,draw,label=270:$8$] {};

\foreach \from/\to in {c1/c2,c1/c3,c3/c4,c3/c5,c5/c6,c5/c7,c7/c8,c7/c9,c9/c10,c9/c11,c9/c12,c9/c13} \draw (\from) -- (\to);
\end{tikzpicture}
\end{center}
\caption{Three examples of linearly ordered phylogenetic trees.}\label{fig-Rgood}
\end{figure}
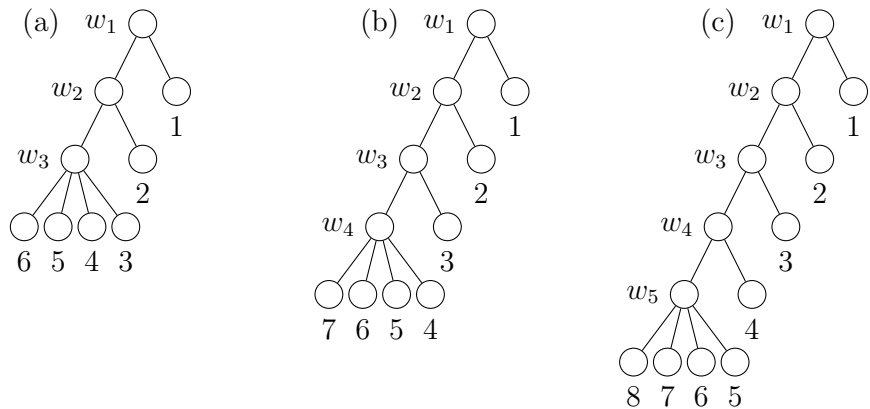

\begin{defn}\label{def:R-good}
Let $G$ be a linearly ordered tree or forest and let $R$ have no exposed odds.  Say $G$ is \emph{$R$-good} if and only if for all vertices $v$, either $v$ is a leaf or $d(v)=2$ or $d(v) \in a(R)$, unless $v$ has $2$-left-odd ancestry, in which case either $v$ is a leaf or $d(v) \in b(R)$. 

For $d \geq 1$ say that $G$ is \emph{$R(d)$-good} if and only if for all vertices $v$, either $v$ is a leaf or $d(v)=s_d(2)~(=d+1)$ or $d(v)=s_d(a)$ for some $a \in a(R)$, unless $v$ has $s_d(2)$-left-odd ancestry, in which case either $v$ is a leaf or $d(v)=s_d(b)$ for some $b \in b(R)$.
\end{defn}

Note that $R$-goodness and $R(1)$-goodness coincide.  When $3 \in R$, the next claim shows that ``has $s_d(2)$-left-odd ancestry'' in Definition \ref{def:R-good} can be replaced by the simpler ``is the left-most child of a vertex $w$ with $d(w)=s_d(2)$.'' So when $3 \in R$, all non-leaf left-children of degree $s_d(2)$ vertices in an $R(d)$-good tree have degree $s_d(n)$ for $n >2$.

\begin{claim} \label{claim:3inR}
If $3 \in R$, then $G$ is $R(d)$-good if and only if for all vertices $v$, either $v$ is a leaf or  $d(v) =s_d(2)$ or $d(v) = s_d(a)$ for some $a \in a(R)$, unless $v$ is the left-most child of a vertex $w$ with $d(w)= s_d(2)$, in which case either $v$ is a leaf or $d(v) = s_d(b)$ for some $b \in b(R)$.
\end{claim}
  
\begin{proof} 
If $3 \in R$, then an $R(d)$-good tree or forest cannot have a vertex $w_2$ as a left-most child of a vertex $w_1$ where $d(w_2)=d(w_1)=s_d(2)$.  Indeed, one of $w_1$ or $w_2$ would have $s_d(2)$-left-odd ancestry, and $2 \notin b(R)$.  
\end{proof}

We provide a few examples to illustrate these definitions.

\begin{ex}
Suppose that $R=\{1,2\} \cup\{4,5,6\}$, so $a(R) = \{4\}$ and $b(R) = \{2,6\}$. Consider the three phylogenetic trees in Figure \ref{fig-Rgood}.  Trees (a) and (c) are $R$-good while tree (b) is not, since vertex $w_4$ has $2$-left-odd ancestry, but $w_4$ is not a leaf and $d(w_4)=4 \notin b(R)$.
\end{ex}

\begin{ex}
If $R=\{1,2\}$, then an $R$-good tree is precisely a binary tree with ordered children and a proper leaf labeling, and an $R(d)$-good tree is precisely a tree with ordered children and all degrees $0$ or $d+1$, together with a proper leaf labeling.
\end{ex}

\begin{ex}
If $R = [r]$ for even $r \geq 4$, then an $R$-good tree is precisely a leaf-labeled tree with ordered children and all degrees $0$, $2$, or $r$ and where the left children of vertices of degree $2$ have degree $0$ or $r$. Note that for $R=[4]$, none of the trees in Figure \ref{fig-Rgood} are $R$-good.
\end{ex}

We define ${\cal T}^{\rm i.o., good}_R(n)$  (${\cal T}^{\rm i.o., good}_{R(d)}(n)$) to be the class of increasingly ordered $R$-good ($R(d)$-good) phylogenetic trees on $n$ leaves and define the analogous notations for other classes of ordered $R$- and $R(d)$-good trees and forests. For example, ${\cal F}_{R(d)}^{\rm m.o., good}(n,k)$ is the set of $R(d)$-good min-first ordered phylogenetic forests with $n$ leaves and $k$ components.  If $d=1$, we write this as just ${\cal F}_{R}^{\rm m.o., good}(n,k)$.  It is straightforward to check that good trees and forests are even. Indeed, since $R$ has no exposed odds, the sets $a(R)$ and $b(R)$ are comprised of even numbers.  By the definition of $R(d)$-goodness, this means the internal sequence (see Definition \ref{def:internal-sequence}) of $G$ is comprised of even numbers and hence has even sum.

Our main results are the following theorems. 
\begin{thm} \label{thm:single-forest}
For all $R \subseteq {\mathbb N}$ with $1 \in R$ and with no exposed odds, and for all $n,k \geq 1$, we have
$$
{n \brace k}^{-1}_{R} = (-1)^{n-k} \left|{\mathcal F}^{\rm i.o., good}_R(n,k)\right|,
$$
$$
{n \brack k}^{-1}_{R} = (-1)^{n-k} \left|{\mathcal F}^{\rm m.o., good}_R(n,k)\right|,
$$
$$
L(n,k)^{-1}_{R} = (-1)^{n-k} \left|{\mathcal F}^{\rm l.o., good}_R(n,k)\right|.
$$
\end{thm}

\begin{thm} \label{thm:single-forest-stretched}
For all $R \subseteq {\mathbb N}$ with $1 \in R$ and with no exposed odds, all $d \geq 1$, and all $n,k \geq 1$, we have
$$
{n \brace k}^{-1}_{R(d)} = (-1)^{(n-k)/d} \left|{\mathcal F}^{\rm i.o., good}_{R(d)}(n,k)\right|,
$$
$$
{n \brack k}^{-1}_{R(d)} = (-1)^{(n-k)/d} \left|{\mathcal F}^{\rm m.o., good}_{R(d)}(n,k)\right|,
$$
$$
L(n,k)^{-1}_{R(d)} = (-1)^{(n-k)/d} \left|{\mathcal F}^{\rm l.o., good}_{R(d)}(n,k)\right|.
$$
\end{thm}
Notice that Theorem \ref{thm:single-forest} is the just the special case $d=1$ of Theorem \ref{thm:single-forest-stretched}. We also note that  Theorem \ref{thm:single-forest-stretched} is vacuously true if  $d \nmid (n-k)$.  In those cases, we will show that ${n \brace k}^{-1}_{R(d)} = {n \brack k}^{-1}_{R(d)} = L(n,k)^{-1}_{R(d)} = 0$ and the forest classes are empty. 

We illustrate these definitions and theorems in the case where $R={\mathbb N}$. An ordered tree $T$ is ${\mathbb N}$-good if and only if every non-leaf vertex has two children, the left-most of which is a leaf. It follows that $|{\cal T}^{\rm i.o., good}_{\mathbb N}(n)|=(n-1)!$, because any of the $(n-1)!$ proper leaf-labelings in which the right-most child of the non-leaf vertex furthest from the root gets label $n$ yields an ${\mathbb N}$-good increasingly ordered tree. On the other hand $|{\cal T}^{\rm m.o., good}_{\mathbb N}(n)|=1$, because for $T$ to be min-first ordered, the leaves must be labeled in increasing order when read counterclockwise from the root. Finally we have $|{\cal T}^{\rm l.o., good}_{\mathbb N}(n)|=n!$, because in this case there is no restriction on the leaf-labeling. See Figure \ref{fig-R=N}.

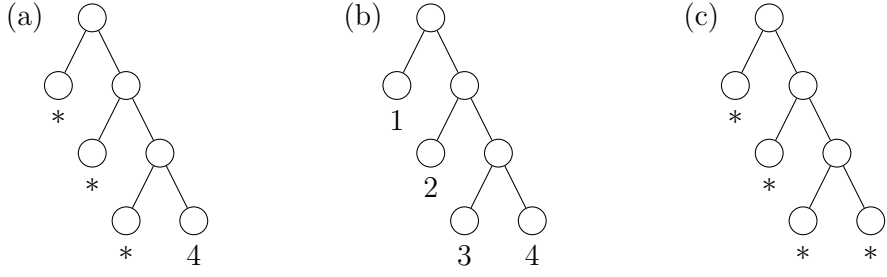
\begin{figure}[ht!]
\begin{center}
\begin{tikzpicture}[scale=.9,every node/.style={scale=0.95}]

\node at (0,5) {(a)};

\node (a1) at (1,5) [circle,draw] {};
\node (a2) at (0.5,4) [circle,draw,label=270:$*$] {};
\node (a3) at (1.5,4) [circle,draw] {};
\node (a4) at (1,3) [circle,draw,label=270:$*$] {};
\node (a5) at (2,3) [circle,draw] {};
\node (a6) at (1.5,2) [circle,draw,label=270:$*$] {};
\node (a7) at (2.5,2) [circle,draw,label=270:$4$] {};

\foreach \from/\to in {a1/a2,a1/a3,a3/a4,a3/a5,a5/a6,a5/a7} \draw (\from) -- (\to);

\node at (5,5) {(b)};
\node (b1) at (6,5) [circle,draw] {};
\node (b2) at (5.5,4) [circle,draw,label=270:$1$] {};
\node (b3) at (6.5,4) [circle,draw] {};
\node (b4) at (6,3) [circle,draw,label=270:$2$] {};
\node (b5) at (7,3) [circle,draw] {};
\node (b6) at (6.5,2) [circle,draw,label=270:$3$] {};
\node (b7) at (7.5,2) [circle,draw,label=270:$4$] {};

\foreach \from/\to in {b1/b2,b1/b3,b3/b4,b3/b5,b5/b6,b5/b7} \draw (\from) -- (\to);

\node at (10,5) {(c)};
\node (c1) at (11,5) [circle,draw] {};
\node (c2) at (10.5,4) [circle,draw,label=270:$*$] {};
\node (c3) at (11.5,4) [circle,draw] {};
\node (c4) at (11,3) [circle,draw,label=270:$*$] {};
\node (c5) at (12,3) [circle,draw] {};
\node (c6) at (11.5,2) [circle,draw,label=270:$*$] {};
\node (c7) at (12.5,2) [circle,draw,label=270:$*$] {};

\foreach \from/\to in {c1/c2,c1/c3,c3/c4,c3/c5,c5/c6,c5/c7} \draw (\from) -- (\to);
\end{tikzpicture}
\end{center}
\caption{(a) $|{\cal T}^{\rm i.o., good}_{\mathbb N}(4)|=(4-1)!$; (b) $|{\cal T}^{\rm m.o., good}_{\mathbb N}(4)|=1$; and (c) $|{\cal T}^{\rm l.o., good}_{\mathbb N}(4)|=4!$.}
\label{fig-R=N}
\end{figure}

Thus Theorem \ref{thm:single-forest} tells us 
\[
{n \brace 1}^{-1}_{\mathbb N} = (-1)^{n-1}|{\cal T}^{\rm i.o., good}_{\mathbb N}| = (-1)^{n-1}{n \brack 1},
\]
\[{n \brack 1}^{-1}_{\mathbb N} = (-1)^{n-1} |{\cal T}^{\rm m.o., good}| = (-1)^{n-1}{n \brace 1},\text{ and}
\]
\[L(n,1)^{-1}_{\mathbb N} = (-1)^{n-1}|{\cal T}^{\rm l.o., good}_{\mathbb N}| = (-1)^{n-1}L(n,1),
\]
matching the first columns of the identities $[{n \brace k}]_{n,k \geq 1}^{-1} = [(-1)^{n-k} {n \brack k}]_{n,k \geq 1}$, etc. 

Some other specific illustrations of these theorems are discussed in Section \ref{sec:conclusion}. 

\section{Preliminary lemmas} \label{sec:preliminaries}

Let $a = (a_n)_{n \geq 1}$ be a sequence of complex numbers with $a_1 \neq 0$. For $n, k \geq 1$ set
\begin{equation} \label{eq:ank-from-an}
a_{n,k} = \sum \left\{ a_{|P_1|} a_{|P_2|} \cdots a_{|P_k|} : \{P_1, \ldots, P_k\} \mbox{ a set partition of } [n] \right\}
\end{equation}
and set 
$$
A_a = [a_{n,k}]_{n,k \geq 1}.
$$
Note that $A_a$ is lower triangular as no partition of $[n]$ has more than $n$ parts, and also that $a_{n,n}=a_1^n$, so that $A_a$ is invertible if and only if $a_1 \neq 0$.

All the $R$-restricted numbers we consider are of the form $a_{n,k}$ for certain choices of $a_n$.  For example, note that ${n \brack k}_R = a_{n,k}$ where $a_n = (n-1)!{\bf 1}_{\{n \in R\}}$.  (Here and throughout we use ${\bf 1}_S$ for the indicator function of the event $S$, the function which takes value $1$ if $S$ occurs and is $0$ otherwise.) This may be seen as follows.  To obtain a partition of $[n]$ into $k$ non-empty cycles of the allowed sizes we first pick a partition of $[n]$ into $k$ non-empty sets $\{P_1, \ldots, P_k\}$ and then for each block $P_i$ choose one of the cycles that may be formed from the elements of $P_i$.  There are $a_{|P_1|}a_{|P_2|} \cdots a_{|P_k|}$ ways of completing the second step: if $P_i$ is of an allowed size, there are $a_{|P_i|} = (|P_i| - 1)!$ possible cycles and otherwise there are $a_{|P_i|} = 0$ possible cycles.  Similarly, ${n \brace k}_R = a_{n,k}$ where $a_n = {\bf 1}_{\{n \in R\}}$, and $L(n,k)_R = a_{n,k}$ where $a_n = n!{\bf 1}_{\{n \in R\}}$.  In all three cases $a_1 \neq 0$ and $A_a$ is invertible if and only if $1 \in R$.

All of our numbers ${n \brace k}_R$, ${n \brace k}^{-1}_R$, etc.\! are thus entries of matrices of the form $A_a$ or $A_a^{-1}$. As we shall see these matrices are submatrices of matrices belonging to the exponential Riordan group.  We now define this group and see that its law of multiplication gives a nice approach to calculating the entries of $A_a^{-1}$.

Given a sequence of complex numbers $f = (f_n)_{n \geq 0}$ we define the {\em exponential generating function of $f$} to be $f(x) = \sum_{n=0}^\infty f_n x^n/n!$. Given $f(x) = \sum_{n=0}^\infty f_n x^n/n!$, let $\ord (f(x)) := \min \{n \geq 0 : f_n \neq 0\}$.  If $f(x)$ and $g(x)$ are exponential generating functions with $\ord(f(x)) = 0$ and $\ord(g(x)) = 1$ then for $k \geq 0$ let $(M_{n,k})_{n \geq 0}$ be the sequence whose exponential generating function is $f(x) g^k(x)/k!$ (that is, $\sum_{n=0}^\infty M_{n,k} x^n/n! = f(x) g^k(x)/k!$). Denote by $[f(x),g(x)]$ the infinite matrix $[M_{n,k}]_{n,k \geq 0}$.  

The \emph{exponential Riordan group} (see e.g. \cite[Chapter 8]{Barry}) is the group of all matrices of the form $[f(x),g(x)]$ with $\ord(f(x)) = 0$ and $\ord(g(x)) = 1$.  The binary operation of this group is matrix multiplication and is computed by $[f(x),g(x)] [u(x),v(x)] = [f(x) u(g(x)), v(g(x))]$.  The identity element is the identity matrix $I = [1,x]$ and $[f(x),g(x)]^{-1} = [1/f(g^{-1}(x)),g^{-1}(x)]$.  Here $g^{-1}(x)$ is the {\em reversion} or  compositional inverse of $g(x)$, the unique power series satisfying $g(g^{-1}(x))=g^{-1}(g(x))=x$.

Let $a(x)$ be the exponential generating function of the sequence $a = (a_n)_{n \geq 1}$. It follows from (\ref{eq:ank-from-an}) and the exponential formula (see e.g.\! \cite[Chapter 3]{Wilf}) that the exponential generating function of the sequence $(a_{n,k})_{n \geq 1}$ of the entries of the $k$th column of $A_a$ is $a^k(x)/k!$.  Thus $A_a = [1,a(x)]_{0,0}$, the matrix obtained by removing the $0$th row and $0$th column of the exponential Riordan matrix $[1,a(x)]$. Note that the exponential generating function of the $0$th column of $[1, a(x)]$ is $1$ so the $(n,0)$ entry of $[1,a(x)]$ is ${\bf 1}_{\{n=0\}}$.  Thus if $b = (b_n)_{n \geq 1}$ is another sequence with $b_1 \neq 0$ and exponential generating function $b(x)$, then 
$$
A_a A_b = [1,a(x)]_{0,0}[1,b(x)]_{0,0} = ([1,a(x)][1,b(x)])_{0,0} = [1,b(a(x))]_{0,0} = A_c
$$ 
where, by the exponential Riordan group multiplication law, $c = (c_n)_{n \geq 1}$ has exponential generating function $b(a(x))$. If $b(x) = a^{-1}(x)$, $A_a A_b = I = [1,x]$. This gives the following fundamental lemma.
\begin{lemma} \label{lem:Riordan} Let $a = (a_n)_{n \geq 1}$ be a sequence of complex numbers with $a_1 \neq 0$ and let $a(x) = \sum_{n=1}^\infty a_n x^n/n!$ be its exponential generating function.  Let \[A_a = [a_{n,k}]_{n,k \geq 1}\] where 
\[a_{n,k} = \sum \left\{ a_{|P_1|} a_{|P_2|} \cdots a_{|P_k|} : \{P_1, \ldots, P_k\} \mbox{ a set partition of } [n] \right\}.\] 
Let $(b_n)_{n \geq 1}$ be the sequence of complex numbers whose exponential generating function is $a^{-1}(x)$. 
Then 
\[A^{-1}_a  = A_b = [b_{n,k}]_{n,k \geq 1}\] 
with 
\begin{equation} \label{eq:building-bnk}
b_{n,k} = \sum \left\{ b_{|P_1|} b_{|P_2|} \cdots b_{|P_k|} : \{P_1, \ldots, P_k\} \mbox{ a set partition of } [n] \right\}.
\end{equation}
\end{lemma}
As an example, we apply this lemma to the case $a_n =1$ in which $A_a = [{n \brace k}]_{n,k \geq 1}$.  Since $a(x) = \exp(x)-1$, we have $b(x) = a^{-1}(x) = \log(1+x) = \sum_{n=1}^\infty (-1)^{n-1} x^n/n$, which is the exponential generating function of $b_n = (-1)^{n-1} (n-1)!$.  A simple calculation shows that the sign of $b_{n,k}$ is $(-1)^{n-k}$ and that $b_{n,k} = (-1)^{n-k} {n \brack k}$.  (See the method of calculation of ${n \brack k}_R$ given in the second paragraph of this section.)  Applying Lemma \ref{lem:Riordan} we obtain the classical result $[{n \brace k}]^{-1}_{n,k \geq 1} = [(-1)^{n-k}{n \brack k}]_{n,k \geq 1}$ alluded to in the introduction. 
The well known inverses of $[{n \brack k}]_{n,k \geq 1}$ and $[L(n,k)]_{n,k\geq 1}$ can be obtained similarly.  

All sequences $a = (a_n)_{n\geq 1}$ that we consider will consist of non-negative integers with $a_1 = 1$. This ensures that the entries of $A_a^{-1}$ are integers, a (perhaps minimum) requirement for a combinatorial interpretation of those entries.  Indeed, if we examine the formula for $a_{n,k}$ in terms of the $a_n$ we see that the matrix $A_a$ will in this case be lower triangular, have integer entries, and have all $1$'s down the diagonal.  Thus, by the co-factor formula for the inverse of a matrix, $A_a^{-1}$ will also have the same three properties.

We will also need the following combinatorial Lagrange inversion formula. If $a_1 \neq 0$ and a $T$ is a phylogenetic tree with $n$ leaves and $m$ non-leaf vertices then we define the {\em $a$-weight} of $T$ to be
$$
w_a(T) = (-1)^m a_1^{-(m+n)}\prod\left\{a_{d(v)}:v \in V(T),~\mbox{$d(v) \neq 0$}\right\}. 
$$
Note that if a tree $T$ consists of just a root then $w_a(T) = 1/a_1$, as the root is considered a leaf. 
The following result has appeared numerous times in the literature.  It is the case $r=1$ of the multi-variable generalization Theorem 3.3.9 of \cite{Ginzburg-Kapranov} and that paper cites earlier occurrences: \cite[Thm. 3.10]{Wright} where it is attributed to Towber and \cite[Thm. 2.13]{MacKay-et-al}.  The Ph.D. theses of Drake and Taylor contain generalizations of the single variable case: \cite[Thm. 1.3.3]{Drake} and \cite[Sec 3.2]{Taylor}. We include a sketch of a proof for completeness.

\begin{lemma} \label{lem:Lagrange-inversion}
If $a(x)=\sum_{n \geq 1} a_n x^n/n!$ (with $a_1\neq 0$) and $a^{-1}(x) = \sum_{n \geq 1} b_n x^n/n!$ then for $n \geq 1$
$$
b_n = \sum_{T \in {\cal T}(n)} w_a(T).
$$
\end{lemma}

\begin{proof} 
Solving $[x^n] (f(\sum_{n=1}^\infty b_n x^n/n!) - x) = 0$ for $b_n$ we get $b_1=1/a_1$ and the recurrence \[b_n = -a_1^{-1}\sum_{k=2}^n a_k \left(\sum _{(i_1, \ldots, i_k)} \frac{1}{k!}\binom{n}{i_1,\ldots,i_k}\prod_{j=1}^k b_{i_j}\right)\] for $n \geq 2$, where $\sum_{(i_1,\ldots,i_k)}$ is a sum over compositions $(i_1,\ldots,i_k)$ of $n$.

If $t_n = \sum_{T \in {\cal T}(n)} w_a(T)$ then $t_n$ satisfies the same initial condition and recurrence.  Indeed $t_1 = 1/a_1 = b_1$. For $n \geq 2$ each tree $T \in {\cal T}(n)$ is uniquely determined by the unordered collection of subtrees $T_1, \ldots, T_k$ rooted at the $k \geq 2$ children of its root.  For such a tree $T$, $w_a(T) = (-a_k/a_1) w_a(T_1) \cdots w_a(T_k)$. If tree $T_j$ has $i_j$ leaves, the sets of leaves of the $T_j$ form an unordered partition of $[n]$ into $k$ parts of sizes $i_1, \ldots, i_k$.  The recurrence follows by summing $w_a(T)$ first over $k$ and then over all such unordered partitions.
\end{proof}

We will use Lemmas \ref{lem:Riordan} and \ref{lem:Lagrange-inversion} to obtain Theorem \ref{thm:forest-difference}. The idea is this: for the first part of Theorem \ref{thm:forest-difference} (${n \brace k}_R^{-1}$) we choose $(a_n)_{n \geq 1}$ so that the matrix $A_a$ in Lemma \ref{lem:Riordan} is precisely $[{n \brace k}_R]_{n, k \geq 1}$. The appropriate choice is $a_n={\bf 1}_{\{n \in R\}}$. Lemma \ref{lem:Lagrange-inversion} allows us to conclude that $b_n$ (the $n$th entry in the first column of the inverse matrix) is a weighted sum of phylogenetic trees, and we argue that this is the same as a signed, but otherwise unweighted, sum of increasingly ordered trees. That is, $b_n$ is the difference between the cardinalities of two explicitly defined sets of increasingly ordered trees. From (\ref{eq:building-bnk}) we then conclude that $b_{n,k}$ is the difference between the cardinalities of two explicitly defined sets of increasingly ordered forests, as claimed. The only change in the approach to the other two parts of Theorem \ref{thm:forest-difference} is the choice of $a_n$.

We conclude this section by briefly discussing our approach to Theorems \ref{thm:single-forest} and \ref{thm:single-forest-stretched}. We discuss only the case $d=1$ here. Suppose that for some $R \subseteq {\mathbb N}$ with $1 \in R$ and with no exposed odds we can find, for each $n$, an involution of ${\mathcal T}^{\rm i.o.}_{R}(n)$ that in its orbits of size $2$ toggles between even and odd trees, and fixes precisely the set of $R$-good trees (which recall are all even; see the paragraph before the statement of Theorem \ref{thm:single-forest}). Using this involution we get from Theorem \ref{thm:forest-difference} (in the special case $k=1$) that ${n \brace 1}_R^{-1} = (-1)^{n-1}|{\mathcal T}^{\rm i.o., good}_{R}(n)|$. But then from Lemma \ref{lem:Riordan} (and in particular equation (\ref{eq:building-bnk})) we get that ${n \brace k}_R^{-1} = (-1)^{n-k}|{\mathcal F}^{\rm i.o., good}_{R}(n,k)|$. The key point here is that $(b_n)_{n \geq 1}$ is an alternating sequence, from which it follows that every summand contributing to the sum defining $b_{n,k}$ contributes the same sign --- $(-1)^{n-k}$ --- something which would not necessarily be the case if $(b_n)_{n \geq 1}$ was not alternating. Analogous phenomena hold for ${n \brack k}_R^{-1}$ and for $L(n,k)_R^{-1}$. So much of our proof will involve finding this involution, which we give in Algorithm \ref{alg:algorithm}, and proving that it has the correct properties, which is done in Lemma \ref{lem:algorithm-facts}. We need to add a little more to this argument to deal with sets of the form $R(d)$; this is also discussed in Section \ref{sec:proofs}.

\section{Proofs} \label{sec:proofs}

\subsection{Proof of Theorem \ref{thm:forest-difference}}

That the inverse matrices under discussion exist if and only if $1 \in R$ is evident. Let $R \subseteq \mathbb{N}$ with $1 \in R$, and let $(a_n)_{n \geq 1}$ be a sequence of non-negative integers with $a_1 = 1$.  Let $A_a^{-1} = [b_{n,k}]_{n,k \geq 1}$ (with the notation following that in Lemma \ref{lem:Riordan}).

For $T \in {\cal T}(n)$ with $m$ non-leaf vertices, $T$ has $n + m - 1$ edges, which we denote by $e(T)$. Adopting the convention $a_0=1$ we get from Lemma \ref{lem:Lagrange-inversion} that $b_n = (-1)^{n-1} \sum_{T \in {\cal T}(n)} N_a(T)$ where 
\[N_a(T) = (-1)^{e(T)} \prod_{v \in V(T)} a_{d(v)}.\] 
Note that if $T$ is turned into a tree with ordered children by assigning to each complete set of $k$ siblings one of $a_k$ possible orderings, then the number of such trees obtainable from $T$ is $|N_a(T)|$.

Let $a_n = {\bf 1}_{\{n \in R\}}$. Then if $T \in {\cal T}(n)$, $|N_a(T)| = 1$ if $T$ has all down-degrees in $R$ and $N_a(T)=0$ otherwise.  Thus 
$$
|{\cal T}_R^{\rm i.o., even}(n)| = \sum_{T \in {\cal T}^{\rm even}(n)} N_a(T)
$$ 
and 
$$
|{\cal T}_R^{\rm i.o., odd}(n)| = -\sum_{T \in {\cal T}^{\rm odd}(n)} N_a(T)
$$ 
as there is precisely one way to turn each $T \in {\cal T}(n)$ with all down-degrees in $R$ into an increasingly ordered tree. Thus 
$$
b_n = (-1)^{n-1}( |{\cal T}^{\rm i.o., even}(n)| - |{\cal T}^{\rm i.o., odd}(n)|).
$$  
We claim that $b_{n,k} = (-1)^{n-k}(|{\cal F}^{\rm i.o., even}(n,k)| - |{\cal F}^{\rm i.o., odd}(n,k)|)$, via equation (\ref{eq:building-bnk}). Indeed, a forest on $n$ leaves with $k$ components can be chosen in two stages. The first stage is to pick a partition $\{P_1, \ldots, P_k\}$ of the label set $[n]$, with say $|P_i| = n_i$. The second stage is to build for each $P_i$ a tree whose leaves are labeled with those $n_i$ labels.  Examine the term $b_{n_1} \cdots b_{n_k}$ of the sum for $b_{n,k}$. Since $(n_1-1) + \cdots (n_k - 1) = n-k$, this term is 
\[
(-1)^{n-k}\left(|{\cal T}^{\rm i.o., even}(n_1)| - |{\cal T}^{\rm i.o., odd}(n_1)|\right) \cdots \left(|{\cal T}^{\rm i.o., even}(n_k)| - |{\cal T}^{\rm i.o., odd}(n_k)|\right).
\]
The internal sequence of a forest has an even (odd) sum if and only if an even (odd) number of its trees have internal sequences with odd sum so $(-1)^{n-k} b_{n_1} \cdots b_{n_k}$ is the number of even forests whose trees have label sets $P_i$ minus the number of odd forests whose trees have label sets $P_i$.  Thus $b_{n,k} = (-1)^{n-k}(|{\cal F}^{\rm i.o., even}(n,k)| - |{\cal F}^{\rm i.o., odd}(n,k)|)$ as claimed.

We turn to the second statement in Theorem \ref{thm:forest-difference}. Let $a_n=(n-1)!{\bf 1}_{\{n \in R\}}$. Then if $T \in {\cal T}(n)$, $|N_a(T)|$ is the number of ways $T$ can be turned into a min-first ordered tree with all down-degrees in $R$.  Note that there are $0$ ways if $T$ has a vertex with down-degree not in $R$.  Thus 
\begin{eqnarray*}
|{\cal T}_R^{\rm m.o., even}(n)| & = & \sum_{T \in {\cal T}^{\rm even}(n)} N_a(T), \\ 
|{\cal T}_R^{\rm m.o., odd}(n)| & = & -\sum_{T \in {\cal T}^{\rm odd}(n)} N_a(T), \\
b_n & = & (-1)^{n-1}( |{\cal T}^{\rm m.o., even}(n)| - |{\cal T}^{\rm m.o., odd}(n)|),
\end{eqnarray*}
and
$$
b_{n,k}  =  (-1)^{n-k}(|{\cal F}^{\rm m.o., even}(n,k)| - |{\cal F}^{\rm m.o., odd}(n,k)|).
$$

Similarly if $a_n=n!{\bf 1}_{\{n \in R\}}$ then 
$$
b_n = (-1)^{n-1}( |{\cal T}^{\rm l.o., even}(n)| - |{\cal T}^{\rm l.o., odd}(n)|
$$ 
and 
$$
b_{n,k} = (-1)^{n-k}(|{\cal F}^{\rm l.o., even}(n,k)| - |{\cal F}^{\rm l.o., odd}(n,k)|).
$$

\subsection{Proofs of Theorem \ref{thm:single-forest-stretched} and Corollary \ref{cor:special-cases}}

Recall that Theorem \ref{thm:single-forest} is the special case $d=1$ of Theorem \ref{thm:single-forest-stretched}, so our focus in this section is Theorem \ref{thm:single-forest-stretched}. 

All the results in Theorem \ref{thm:single-forest-stretched} are obtained as follows.  We define an involution on increasingly (min-first, linearly) ordered phylogenetic trees with down-degrees in $R$ (or $R(d)$) that maps odd trees to even trees and vice versa and we show that the trees that are fixed by this involution are precisely the $R$-good ($R(d)$-good) trees in that class. Since good trees are even this means that $b_n = (-1)^{n-1} |{\cal T}^{\rm i.o., good}_R(n)|$ and $b_{n,k} = (-1)^{n-k} |{\cal F}^{\rm i.o., good}_R(n,k)|$, (or $b_n = (-1)^{(n-1)/d} |{\cal T}^{\rm i.o., good}_{R(d)}(n)|$ and $b_{n,k} = (-1)^{(n-k)/d} |{\cal F}^{\rm i.o., good}_{R(d)}(n,k)|$), etc.

The image of a tree under this involution, whether the tree is increasingly, min-first, or linearly ordered, is obtained by applying the \emph{same} algorithm, Algorithm \ref{alg:algorithm} below. We will describe this algorithm and derive its properties for general $d$. The algorithm is expressed in terms of $s_d(n)=d(n-1)+1$.  Since $s_1(n)=n$, the special case $d=1$ of both the algorithm and the analysis can be recovered by reading ``$s_d(n)$'' throughout as ``$n$''. 

\begin{alg} \label{alg:algorithm}
Let $R \subseteq \mathbb{N}$ with $1 \in R$ have no exposed odds and let $d \geq 1$.

Input: A tree $T$ in ${\cal T}^{\rm i.o.}_{R(d)}$ (${\cal T}^{\rm m.o.}_{R(d)}$, ${\cal T}^{\rm l.o.}_{R(d)}$).

Output: A tree $A(T)$ in ${\cal T}^{\rm i.o.}_{R(d)}$ (${\cal T}^{\rm m.o.}_{R(d)}$, ${\cal T}^{\rm l.o.}_{R(d)}$, respectively).

\begin{enumerate}

\item (Initial phase.) Let $v_1, v_2, \ldots, v_k$ be the unique right-most path in $T$ from root $v_1$ to leaf $v_k$, i.e. $v_{j+1}$ is the right-most child of $v_j$ for $1 \leq j <k$. Consider each vertex $v_j$ in this path in increasing order of $j$ for $1 \leq j \leq k-1$.

\begin{enumerate} 

\item If $v_j$ has $s_d(2)$ children, let $v'_j$ be the left-most child of $v_j$. If $v'_j$ is not a leaf and $v'_j$ has $s_d(n)$ children for $n \not \in b(R)$, remove vertex $v'_j$ and all edges adjacent to it and then make every child of $v'_j$ a child of $v_j$.   The vertex $v_j$ now has $s_d(n+1)$ children: the former children of $v'_j$ and the original $d$ right-most children of $v_j$.  Linearly order these children as follows.  Let each set of children inherit the linear ordering they had originally and place the former children of $v'_j$ before the original $d$ right-most children of $v_j$. Leave the orderings on the children of all other vertices $v \neq v_j, v'_j$ unchanged.  Let $A(T)$ be the resulting tree.

\item If $v_j$ has $s_d(n)$ children for $n > 2$ where $n \not \in a(R)$: remove the edges between $v_j$ and its left-most $s_d(n-1)$ children, create a new vertex $v'_j$ to be the parent of these children, and make $v'_j$ a child of $v_j$.  Let the $s_d(n-1)$ children of $v'_j$ inherit the linear ordering they were assigned as children of $v_j$.  Make $v'_j$ be the left-most child of $v_j$ and let the other $d$ children of $v_j$ retain the linear ordering they had before.  Now $d(v_j) = s_d(2)$. Leave the orderings on the children of all other vertices $v \neq v_j, v'_j$ unchanged. Let $A(T)$ be the resulting tree.

\end{enumerate}

\item (Recursive phase.) Suppose now that for all $1 \leq j \leq k-1$, $v_j$ fails both criteria in step 1.  Remove $v_1, \ldots, v_k$ and all edges adjacent to these vertices.  If $v_j$ has $s_d(2) = d+1$ children, also remove the left-most child $v'_j$ of $v_j$ and all edges adjacent to $v'_j$.  This leaves behind a possibly empty forest $F$.

If $F$ is not empty consider its component trees $T'$ in increasing order of the $\ell_{\max}$ label on their root (or the $\ell_{\min}$ label if we are dealing with min-first ordered trees). If there is any tree $T'$ for which the algorithm, when applied to $T'$, would produce a tree $A(T')\neq T'$ then replace the first such $T'$ in $T$ by $A(T')$ and let $A(T)$ be the resulting tree.

If $F$ is empty, or if the algorithm would fix each tree $T'$ in $F$, let $A(T) = T$.

\end{enumerate}
\end{alg}

Note that in the recursive phase the component trees $T'$ are not necessarily {\em properly} leaf-labeled. By ``apply the algorithm to $T'$'' what we formally mean is ``for each $i$ replace the $i$th largest leaf label of $T'$ with the label $i$, to obtain a new tree $T''$ that is properly labeled; then apply the algorithm to $T''$; and then, for each $i$, replace the label $i$ in $A(T'')$ with the $i$th largest leaf label of $T'$''.

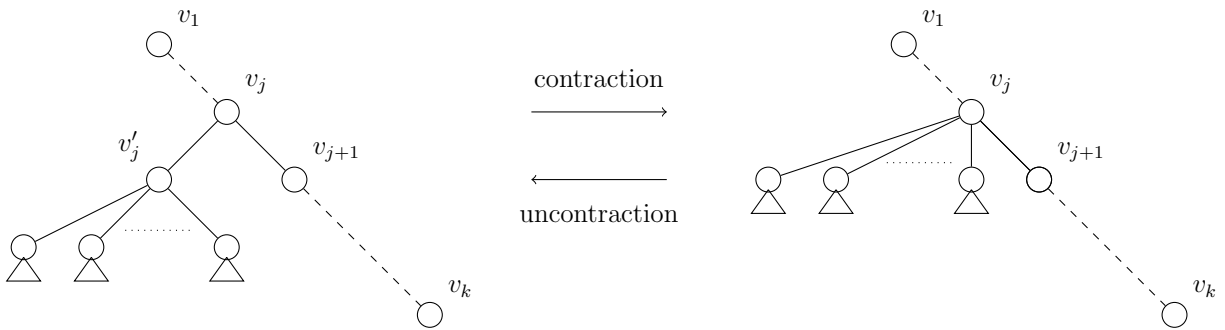
\begin{figure}[ht!]
\begin{center}

\begin{tikzpicture}[scale=.9,every node/.style={scale=0.85}]

\node (v1) at (0,-2) [circle,draw, label=45:$v_1$] {};
\node (v2) at (1,-3) [circle,draw, label=45:$v_j$] {};
\node (v3) at (2,-4) [circle,draw, label=45:$v_{j+1}$] {};
\node (v4) at (4,-6) [circle,draw, label=45:$v_k$] {};
\node (v5) at (0,-4) [circle,draw, label=135:$v'_j$] {};
\node (v6) at (-2,-5) [circle,draw] {};
\node (v7) at (-1,-5) [circle,draw] {};
\node (v8) at (1,-5) [circle,draw] {};

\foreach \from/\to in {v2/v3,v2/v5,v5/v6,v5/v7,v5/v8} \draw (\from) -- (\to);

\foreach \from/\to in {v1/v2,v3/v4} \draw[dashed] (\from) -- (\to);

\draw[dotted] (-.5,-4.75) -- (.5,-4.75);

\draw (-2,-5.15) -- (-2.25,-5.5) -- (-1.75,-5.5) -- (-2,-5.15);
\draw (-1,-5.15) -- (-1.25,-5.5) -- (-.75,-5.5) -- (-1,-5.15);
\draw (1,-5.15) -- (.75,-5.5) -- (1.25,-5.5) -- (1,-5.15);

\draw [->] (5.5,-3) -- (7.5,-3);
\draw [<-] (5.5,-4) -- (7.5,-4);
\node at (6.5,-2.5) {contraction};
\node at (6.5,-4.5) {uncontraction};

\node (v1) at (11,-2) [circle,draw, label=45:$v_1$] {};
\node (v2) at (12,-3) [circle,draw, label=45:$v_j$] {};
\node (v3) at (13,-4) [circle,draw, label=45:$v_{j+1}$] {};
\node (v4) at (15,-6) [circle,draw, label=45:$v_k$] {};
\node (v5) at (12,-4) [circle,draw] {};
\node (v6) at (9,-4) [circle,draw] {};
\node (v7) at (10,-4) [circle,draw] {};
\node (v8) at (13,-4) [circle,draw] {};

\foreach \from/\to in {v2/v3,v2/v5,v2/v6,v2/v7,v2/v8} \draw (\from) -- (\to);

\foreach \from/\to in {v1/v2,v3/v4} \draw[dashed] (\from) -- (\to);

\draw[dotted] (10.75,-3.75) -- (11.75,-3.75);

\draw (9,-4.15) -- (8.75,-4.5) -- (9.25,-4.5) -- (9,-4.15);
\draw (10,-4.15) -- (9.75,-4.5) -- (10.25,-4.5) -- (10,-4.15);
\draw (12,-4.15) -- (11.75,-4.5) -- (12.25,-4.5) -- (12,-4.15);

\end{tikzpicture}

\end{center}
\caption{Contraction and uncontraction in the initial phase ($d=1$).}
\label{fig-contract_uncontract}
\end{figure}
We refer to the operation in (1a) as {\em contraction} at $v_j$ {\em towards} $v'_j$, because it corresponds to the usual graph-theoretic operation of contracting the edge $v_j v'_j$.  We refer to the operation in (1b) as {\em uncontraction} at $v_j$ {\em away from} $v_{j+1}$. See Figure \ref{fig-contract_uncontract}.

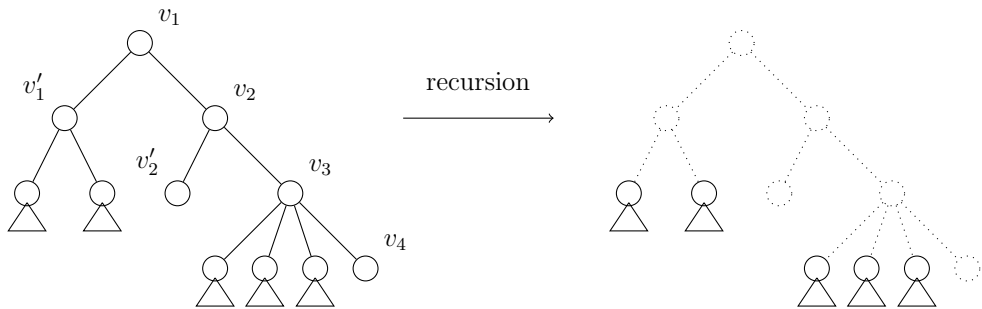
\begin{figure}[ht!]
\begin{center}
\begin{tikzpicture}[scale=1,every node/.style={scale=0.85}]
\node (v1) at (0,0) [circle,draw,label=45:$v_1$] {};
\node (v2) at (1,-1) [circle,draw,label=45:$v_2$] {};
\node (v3) at (2,-2) [circle,draw,label=45:$v_3$] {};
\node (v4) at (3,-3) [circle,draw, label=45:$v_4$] {};
\node (v5) at (-1,-1) [circle,draw, label=135:$v'_1$] {};
\node (v6) at (-1.5,-2) [circle,draw] {};
\node (v7) at (-.5,-2) [circle,draw] {};
\node (v8) at (.5,-2) [circle,draw, label=135:$v'_2$] {};
\node (v9) at (1,-3) [circle,draw] {};
\node (v10) at (1.66,-3) [circle,draw] {};
\node (v11) at (2.33,-3) [circle,draw] {};

\foreach \from/\to in {v1/v2,v2/v3,v3/v4,v1/v5,v5/v6,v5/v7,v2/v8,v3/v9,v3/v10,v3/v11} \draw (\from) -- (\to);

\draw (-1.5,-2.125) -- (-1.75,-2.5) -- (-1.25,-2.5) -- (-1.5,-2.125);
\draw (-.5,-2.125) -- (-.75,-2.5) -- (-.25,-2.5) -- (-.5,-2.125);
\draw (1,-3.125) -- (.75,-3.5) -- (1.25,-3.5) -- (1,-3.125);
\draw (1.66,-3.125) -- (1.41,-3.5) -- (1.91,-3.5) -- (1.66,-3.125);
\draw (2.33,-3.125) -- (2.08,-3.5) -- (2.58,-3.5) -- (2.33,-3.125);

\draw [->] (3.5,-1) -- (5.5,-1);
\node at (4.5,-.5) {recursion};

\node (v1) at (8,0) [draw, dotted, circle] {};
\node (v2) at (9,-1) [draw,dotted,circle] {};
\node (v3) at (10,-2) [draw,dotted,circle] {};
\node (v4) at (11,-3) [circle,dotted,draw] {};
\node (v5) at (7,-1) [draw,dotted,circle] {};
\node (v6) at (6.5,-2) [circle,draw] {};
\node (v7) at (7.5,-2) [circle,draw] {};
\node (v8) at (8.5,-2) [draw, dotted,circle] {};
\node (v9) at (9,-3) [circle,draw] {};
\node (v10) at (9.66,-3) [circle,draw] {};
\node (v11) at (10.33,-3) [circle,draw] {};

\foreach \from/\to in {v1/v2,v2/v3,v3/v4,v1/v5,v5/v6,v5/v7,v2/v8,v3/v9,v3/v10,v3/v11} \draw[dotted] (\from) -- (\to);

\draw (6.5,-2.125) -- (6.25,-2.5) -- (6.75,-2.5) -- (6.5,-2.125);
\draw (7.5,-2.125) -- (7.25,-2.5) -- (7.75,-2.5) -- (7.5,-2.125);
\draw (9,-3.125) -- (8.75,-3.5) -- (9.25,-3.5) -- (9,-3.125);
\draw (9.66,-3.125) -- (9.41,-3.5) -- (9.91,-3.5) -- (9.66,-3.125);
\draw (10.33,-3.125) -- (10.08,-3.5) -- (10.58,-3.5) -- (10.33,-3.125);

\end{tikzpicture}
\end{center}
\caption{Recursive phase ($d=1$, $R = \{1,2\} \cup \{4,5,6\}$).}
\label{fig-recursive-phase}
\end{figure}

See Figure \ref{fig-recursive-phase} for an example of the recursive phase of the algorithm when $d=1$ and $R = \{1,2\} \cup \{4,5,6\}$.  The right-most path is $v_1, v_2, v_3, v_4$.  Since $d(v_1)=2$ and $d(v'_1) = 2 \in b(R) = \{2,6\}$, since $d(v'_2) = 0$, and since $d(v_3) = 4 \in a(R) = \{4\}$, the algorithm cannot perform an operation in the initial phase.  It removes $v_1,\ldots,v_4$ and $v'_1,v'_2$ and recursively evaluates trees in the resulting forest.

We establish some useful facts about Algorithm \ref{alg:algorithm} in Lemmas \ref{lemma:characterization} and \ref{lem:algorithm-facts}, after which the proof of Theorem \ref{thm:single-forest-stretched} will be quite short.

\begin{lemma} \label{lemma:characterization}
Suppose that $T$ is a tree that produces a forest $F$ via the recursive phase of Algorithm \ref{alg:algorithm}, and let $v$ be a vertex in $F$.  Then $v$ has $s_d(2)$-left-odd ancestry in $T$ if and only if $v$ has $s_d(2)$-left-odd ancestry in $F$.
\end{lemma}
\begin{proof} 
If $v$ has $s_d(2)$-left-odd ancestry in $T$, then in $T$ we have $v_1,\ldots,v_k=v$, $d(v_i)=s_d(2)$ ($1 \leq i < k$), $v_{i+1}$ a left-most child of $v_i$ ($1\leq i < k$), $k$ even, and $v_1$ not a left-most child of a vertex $w$ with $d(w)=s_d(2)$. If $v_1$ is not deleted then these properties hold for $v$ in $F$ exactly as in $T$.  If $v_1$ is deleted it must be in the right-most path. (As $v_1$ is not a left-most child of a vertex with degree $s_d(2)$, it is not one of the $v'_j$.) Thus $v_2$ is also deleted in the recursive phase, and so $v_3$ is a root in $F$.  The path $v_3,\ldots,v_k=v$ then demonstrates that $v$ has $s_d(2)$-left-odd ancestry in $F$. 

For the converse, suppose $v$ in $F$ has $s_d(2)$-left-odd ancestry. So in $F$ we have a path $P$ on $v_1,\ldots,v_k=v$, where $d(v_i)=s_d(2)$ and $v_{i+1}$ is a left-most child of $v_i$ for $1\leq i < k$, $k$ is even, and $v_1$ is not a left-most child of a vertex $w$ with $d(w)=s_d(2)$.  Let $T'$ be the tree in $F$ containing $v_1$.  If $v_1$ is not the root of $T'$ then $P$ is a witness that $v$ has $s_d(2)$-left-odd ancestry in $T$.  Suppose now that $v_1$ is the root of $T'$.  Let $v_j$ be the ancestor of $v_1$ on the right-most path in $T$ that is closest to $v_1$.  If $d(v_j) = s_d(n)$ with $n >2$ then $v_1$ must be a child of $v_j$ and $P$ is a witness that $v$ has $s_d(2)$-left-odd ancestry in $T$.

Now suppose $d(v_j) = s_d(2)$.  The vertex $v_1$ cannot be $v'_j$ or $v_{j+1}$ as these vertices are deleted.  If $v_1$ is some other child of $v_j$ then $P$ is a witness of $s_d(2)$-left-odd ancestry.  So suppose that $v_1$ is a child of $v'_j$.  If $d(v'_j) = s_d(n)$ for $n > 2$ or if $d(v'_j) = s_d(2)$ and $v_1$ is not the left-most child of $v'_j$ then $P$ is a witness of $s_d(2)$-left-odd ancestry.  The remaining possibility is that $d(v'_j) = s_d(2)$ and $v_1$ is the left-most child of $v'_j$.  In this case $v_j,v'_j,v_1, \ldots, v_k$ is a witness that $v$ has $s_d(2)$-left-odd ancestry, as $v_j$ is on the right-most path so cannot be a left-most child of any vertex. 
\end{proof}

\begin{lemma} \label{lem:algorithm-facts} 
Algorithm \ref{alg:algorithm} has the following properties.
\begin{enumerate}

\item[1.] We have $A(T) = T$ if and only if $T$ is an $R(d)$-good tree.

\item[2.] All non-zero degrees in $A(T)$ are in $R(d)$.

\item[3.] If the input tree $T$ is increasingly (min-first, linearly) ordered then so is the output tree $A(T)$.

\item[4.] We have $A(A(T)) = T$ for all $T$.

\end{enumerate}
\end{lemma}

\begin{proof}
We prove these statements by induction on $n$, the number of leaves.

The base case of the induction, $n=1$, is trivial, as ${\cal T}^{\rm i.o.}(1)$, ${\cal T}^{\rm m.o.}(1)$, and ${\cal T}^{\rm l.o.}(1)$ each consist of a single $R(d)$-good tree, an isolated root with label $1$, and the algorithm fixes that tree.  Suppose now that $n > 1$. 

We now show item 1.  Suppose that $T$ is an $R(d)$-good tree.  Let $v$ be a non-leaf vertex on the right-most path of $T$.  By definition of $s_d(2)$-left-odd ancestry, $v$ does not have $s_d(2)$-left-odd ancestry and so either $d(v)=s_d(2)$ or $d(v) = s_d(a)$ for some $a \in a(R)$.  If $d(v) = s_d(2)$ the left-most child of $v$ has $s_d(2)$-left-odd ancestry and so is either a leaf or has degree $s_d(b)$ for some $b \in b(R)$.  Therefore Algorithm \ref{alg:algorithm} proceeds to the recursive phase. This removes the vertices in the right-most path, and the left-children of vertices with degree $d(v) = s_d(2)$. This leaves behind a possibly empty forest $F$. Note that a vertex that remains in $F$ has the same down-degree as in $T$, and the property of having $s_d(2)$-left-odd ancestry transfers to vertices in $F$ by Lemma \ref{lemma:characterization}.
So each component tree meets the definition of being $R(d)$-good, and so by induction is fixed by Algorithm \ref{alg:algorithm}.  Therefore tree $T$ is fixed by Algorithm \ref{alg:algorithm}, i.e., $A(T)=T$.

Conversely, suppose that $A(T)=T$.  Then Algorithm \ref{alg:algorithm} proceeds to the recursive phase, and so all non-leaf vertices on the right-most path must have degree $s_d(a)$ for some $a \in a(R)$, or $s_d(2)$ with the left-child a leaf or having degree $s_d(b)$ for some $b \in b(R)$. The deletion leaves a forest $F$, which by hypothesis has $A(T')=T'$ for each component $T'$ of $F$, and so by induction consists of $R(d)$-good trees $T'$.  So by definition of $R(d)$-goodness a vertex $v$ in $F$ is either a leaf or $d(v) = s_d(2)$ or $d(v) = s_d(a)$ for some $a \in a(R)$, unless $v$ has $s_d(2)$-left-odd ancestry in $F$ in which case $d(v) = s_d(b)$ for some $b \in b(R)$.  In the last case Lemma \ref{lemma:characterization} shows that $v$ has $s_d(2)$-left-odd ancestry in $T$.  Combining this with the fact that down-degrees of $v$ in $F$ are the down-degrees of $v$ in $T$, this shows that $T$ is $R(d)$-good. 

We now show items 2 and 3 in the case that $T$ is produced from a contraction at vertex $v_j$ in step 1(a) of the algorithm.  Suppose $d(v'_j) = s_d(n)$.  Since $n \not \in b(R)$, $n+1 \in R$ and $d(v_j) = s_d(n+1) \in R(d)$ in $A(T)$. All other vertices of $A(T)$ are unchanged from $T$, so  $A(T)$ has all down-degrees in $R(d)$ and we have item 2.

We now show item 3.  Let $v''_1, \ldots v''_m$ and $v'_j, w_2, \ldots, w_d, v_{j+1}$ be the ordered lists of children of $v'_j$ and $v_j$ in $T$ respectively. The ordered list of children of $v_j$ in $A(T)$ is $v''_1, \ldots, v''_m, w_2, \ldots, w_d, v_{j+1}$.

If $T$ is increasingly ordered then \[\ell_{\max}(v''_1) < \cdots < \ell_{\max}(v''_m) \mbox{ and } \ell_{\max}(v'_j) < \ell_{\max}(w_2) < \cdots < \ell_{\max}(w_d) < \ell_{\max}(v_{j+1}).\]   Since \[\ell_{\max}(v'_j) = \max(\ell_{\max}(v''_1), \ldots, \ell_{\max}(v''_m)) = \ell_{\max}(v''_m)\] we have \[\ell_{\max}(v''_1) < \cdots < \ell_{\max}(v''_m) < \ell_{\max}(w_2) < \cdots < \ell_{\max}(w_d) < \ell_{\max}(v_{j+1})\] and thus the children of $v_j$ are increasingly ordered in $A(T)$.  Since the orderings of all other children in $A(T)$ are unchanged from their ordering in $T$, $A(T)$ is increasingly ordered.

If $T$ is min-first ordered then $v''_1$ has the smallest $\ell_{\min}$ label amongst $v''_1, \ldots, v''_m$ and $v'_j$ has the smallest $\ell_{\min}$ label amongst $v'_j, w_2, \ldots, w_m, v_{j+1}$. Since in $T$ we have \[\ell_{\min}(v''_1) = \min(\ell_{\min}(v''_1), \ldots, \ell_{\min}(v''_m)) = \ell_{\min}(v'_j),\] $v''_1$ has the smallest $\ell_{\min}$ label amongst the children of $v_j$ in $A(T)$. Thus $A(T)$ is min-first ordered.

There are no restrictions on the linear orderings in a linearly ordered tree so if $T$ is linearly ordered then $A(T)$ is automatically linearly ordered.

We now show items 2 and 3 in the case that $A(T)$ is produced from $T$ by an uncontraction at vertex $v_j$ in step 1(b) of the algorithm. Suppose $d(v_j) = s_d(n)$ with $n>2$.  Since $n \not \in a(R)$, $n-1 \in R$ and, in $A(T)$, $d(v_j) = s_d(n-1) \in R(d)$ and $d(v_j) =s_d(2)$. It follows that $A(T)$ has all down-degrees in $R(d)$, and we have item 2. 

We now show item 3.  Let $v''_1, \ldots v''_m, w_2, \ldots, w_d, v_{j+1}$ be the ordered list of children of $v_j$ in $T$ where $m = s_d(n-1)$. 

If $T$ is increasingly ordered then \[\ell_{\max}(v''_1) < \cdots < \ell_{\max}(v''_m) < \ell_{\max}(w_2) < \cdots < \ell_{\max}(w_d) < \ell_{\max}(v_{j+1}).\]  Thus the children $v''_1, \ldots v''_m$ of $v'_j$ in $A(T)$ are increasingly ordered. Since \[\ell_{\max}(v'_j) = \max(\ell_{\max}(v''_1), \ldots, \ell_{\max}(v''_m)) = \ell_{\max}(v''_m)\] in $A(T)$, the children of $v_j$ are increasingly ordered in $A(T)$: \[\ell_{\max}(v'_j) < \ell_{\max}(w_2) < \cdots < \ell_{\max}(w_d) < \ell_{\max}(v_{j+1}).\]  Thus, as before, $A(T)$ is increasingly ordered.

If $T$ is min-first ordered then $v''_1$ has the smallest $\ell_{\min}$ label amongst $v''_1, \ldots, v''_m, w_2, \ldots, w_d$, and $v_{j+1}$.  Thus the children of $v'_j$ in $A(T)$ are min-first ordered. Since 
$$
\ell_{\min}(v'_j) = \min(\ell_{\min}(v''_1), \ldots, \ell_{\min}(v''_m)) = \ell_{\min}(v''_1),
$$ 
the children of $v_j$ are min-first ordered in $A(T)$ as well and $A(T)$ is min-first ordered.  As before, if $T$ is a linearly ordered tree then $A(T)$ is automatically linearly ordered. 

If $A(T)$ is produced by a contraction/uncontraction at vertex $v_j$ in step 1, we have shown that $A(T)$ has all down-degrees in $R(d)$, and so we can apply Algorithm \ref{alg:algorithm} to $A(T)$. In this case, we will now show that $A(A(T)) = T$.

In $A(T)$ the right-most path is exactly as it was in $T$. Further, for $i < j$  the number of children of $v_i$ remains unchanged from $T$ to $A(T)$, as does the left-most child of $v_i$ and its children. Since this is the data that determines whether a contraction/uncontraction is to be performed at $v_i$, it follows that if the algorithm is applied to $A(T)$, it does not call for contraction/uncontraction at $v_i$ for any $i < j$. However, at $v_j$, if in $T$ we performed a contraction, then the algorithm calls for an uncontraction at $v_j$ in $A(T)$, while if in $T$ we performed an uncontraction, then the algorithm calls for a contraction at $v_j$ in $A(T)$. In either case, we have $A(A(T))=T$, which gives item 4 in this case.

We now suppose that $A(T)$ is produced by step 2, the recursive phase of the algorithm. If $A(T)=T$, the results are immediate.  Therefore we assume that $A(T) \neq T$, and so there is a $T'$ in $F$ with $A(T') \neq T'$, and $A(T)$ is obtained from $T$ by replacing $T'$ with $A(T')$.  Thus by the initial phase and induction, $A(T)$ has all down-degrees in $R(d)$ and remains increasingly (min-first, linearly) ordered. This gives items 2 and 3. 

Finally we show item 4 in the case where $A(T)$ is produced by step 2. The right-most path stays the same from $T$ to $A(T)$.  Every vertex $v_j$ on the path keeps the children in $A(T)$ it had in $T$ and if $d(v_j) = s_d(2)$, then its left-most child $v'_j$ keeps the children in $A(T)$ it had in $T$.  Thus when the algorithm is applied to $A(T)$ it also produces the same forest in the recursive phase. The collection of subtrees examined when applying the algorithm to $A(T)$ is the same one examined when applying the algorithm to $T$, except that $T'$ has become $A(T')$. The ordering on subtrees remains unchanged, so now $A(T')$ is the first component that is not ${R(d)}$-good. By induction $A(A(T'))=T'$, so $A(A(T))=T$.
\end{proof}

\begin{ex}
Let $R=\{1,2\} \cup \{4,5,6\}$ and consider the tree $T$ in Figure \ref{fig-A(T)} below, which was the only tree in Figure \ref{fig-Rgood} that was not $R$-good (here we use linear ordering).  In this case $T$ contracts edge $w_3w_4$ via Algorithm \ref{alg:algorithm} to produce $A(T)$. Notice also that Algorithm \ref{alg:algorithm} applied to $A(T)$ shows $A(A(T))=T$.
\end{ex}

\begin{figure}[ht!]
\begin{center}
\begin{tikzpicture}[scale=.9,every node/.style={scale=0.95}]

\node at (4,5) {$T$:};
\node (b1) at (6,5) [circle,draw,label=180:$w_1$] {};
\node (b2) at (6.5,4) [circle,draw,label=270:$1$] {};
\node (b3) at (5.5,4) [circle,draw,label=180:$w_2$] {};
\node (b4) at (6,3) [circle,draw,label=270:$2$] {};
\node (b5) at (5,3) [circle,draw,label=180:$w_3$] {};
\node (b6) at (5.5,2) [circle,draw,label=270:$3$] {};
\node (b7) at (4.5,2) [circle,draw,label=180:$w_4$] {};
\node (b8) at (5.25,1) [circle,draw,label=270:$4$] {};
\node (b9) at (4.75,1) [circle,draw,label=270:$5$] {};
\node (b10) at (4.25,1) [circle,draw,label=270:$6$] {};
\node (b11) at (3.75,1) [circle,draw,label=270:$7$] {};

\foreach \from/\to in {b1/b2,b1/b3,b3/b4,b3/b5,b5/b6,b5/b7,b7/b8,b7/b9,b7/b10,b7/b11} \draw (\from) -- (\to);

\node at (9,5) {$A(T)$:};
\node (c1) at (11,5) [circle,draw,label=180:$w_1$] {};
\node (c2) at (11.5,4) [circle,draw,label=270:$1$] {};
\node (c3) at (10.5,4) [circle,draw,label=180:$w_2$] {};
\node (c4) at (11,3) [circle,draw,label=270:$2$] {};
\node (c5) at (10,3) [circle,draw,label=180:$w_3$] {};
\node (c6) at (11,2) [circle,draw,label=270:$3$] {};
\node (c7) at (10.5,2) [circle,draw,label=270:$4$] {};
\node (c8) at (10,2) [circle,draw,label=270:$5$] {};
\node (c9) at (9.5,2) [circle,draw,label=270:$6$] {};
\node (c10) at (9,2) [circle,draw,label=270:$7$] {};

\foreach \from/\to in {c1/c2,c1/c3,c3/c4,c3/c5,c5/c6,c5/c7,c5/c8,c5/c9,c5/c10} \draw (\from) -- (\to);
\end{tikzpicture}
\end{center}
\caption{Tree $T$ produces $A(T)$ via Algorithm \ref{alg:algorithm}.}
\label{fig-A(T)}
\end{figure}
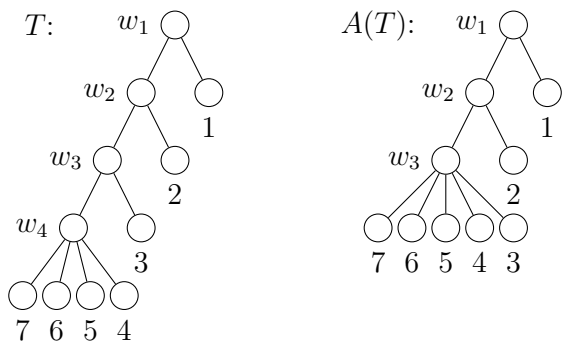

\begin{proof}[Proof of Theorem \ref{thm:single-forest-stretched}]
As noted after the statement of Theorem \ref{thm:single-forest-stretched}, if $\{P_1, \ldots, P_k\}$ is a partition of $[n]$ with part sizes restricted to lie in $R(d)$, then $n = |P_1| + \cdots + |P_k| = s_d(m_1) + \cdots + s_d(m_k) = d(m_1 + \cdots + m_k - k) + k$, so ${n \brace k}_{R(d)} ={n \brack k}_{R(d)}  = L(n,k)_{R(d)} = 0$ if $d \nmid (n-k)$.  Similarly, ${\cal F}_{R(d)}(n,k) = \emptyset$ if $d \nmid (n-k)$. Indeed, let $F$ be a phylogenetic forest with $n$ leaves, $k$ components, and down-degrees in $R(d)$, i.e. $F$ has $m$ non-leaf vertices $v_i$ with $d(v_i) = s_d(n_i)$.  The number of edges of $F$ is $(n+m)-k=\sum_{i=1}^m \left( d(n_i-1)+1 \right)$, giving $n-k=d(n_1+\cdots+n_m-m)$. For this reason in the sequel we only consider triples $(d,n,k)$ with $d \mid n-k$.

If $T$ has internal sequence $(n_i)_{i=1}^m$ then, by case $k=1$ of the edge count of $F$ in the previous paragraph, it has $m = -(n-1)/d + \sum_{i=1}^m n_i$ internal vertices.  
Thus, from Lemma \ref{lem:Lagrange-inversion}, we have
$$
b_n = (-1)^{(n-1)/d} \sum_{T \in {\cal T}(n)} (-1)^{\sum_{i=1}^m n_i} \prod_{i=1}^m a_{d(v_i)}
$$ 
where $v_1, \ldots, v_m$ is the set of non-leaf vertices of the index tree $T$ in the summation.

We begin with the first statement of Theorem \ref{thm:single-forest-stretched}.  
If $a_n = {\bf 1}_{\{n \in R(d)\}}$ then by the method of Theorem \ref{thm:forest-difference} we have 
$$
b_n = b_{n,1} = (-1)^{(n-1)/d}(|{\cal T}^{\rm i.o., even}_{R(d)}(n)| - |{\cal T}^{\rm i.o., odd}_{R(d)}(n)|).
$$
Indeed, since $a_n = {\bf 1}_{\{n \in R(d)\}}$, the $\prod_{i=1}^m a_{d(v_i)}$ factor of the summand is the number of ways of turning the index tree $T$, a properly labeled phylogenetic tree with unordered children, into an increasingly ordered tree with down-degrees in $R(d)$.  The $(-1)^{\sum_{i=1}^m n_i}$ factor of the summand ensures even trees are counted positively and odd trees negatively.

When $A(T) \neq T$ the internal sequence of $A(T)$ is obtained from the internal sequence of $T$ by replacing a pair of indices $n_i,n_j$ with a single entry $n_i + n_j - 1$ or vice versa.  Thus $A(T)$ is even when $T$ is odd and vice versa.  Since $A(T)  = T$ if and only if $T$ is $R(d)$-good and since $R(d)$-good trees are even we get 
$$
b_n = (-1)^{(n-1)/d}|{\cal T}^{\rm i.o., good}_{R(d)}(n)|.
$$  
Using that $b_n$ is alternating along the arithmetic progression $\{1,d+1,2d+1,\ldots\}$ it is easy to check that in this case the sign of all the summands on the right-hand side of equation (\ref{eq:building-bnk}) is $(-1)^{(n-k)/d}$, and so 
$$
b_{n,k} = (-1)^{(n-k)/d}|{\cal F}^{\rm i.o., good}_{R(d)}(n,k)|,
$$ 
as claimed.

For the second statement of Theorem \ref{thm:single-forest-stretched}, if we take $a_n = (n-1)!{\bf 1}_{\{n \in R(d)\}}$ then we get 
$$
b_n = (-1)^{(n-1)/d}(|{\cal T}^{\rm m.o., even}_{R(d)}(n)| - |{\cal T}^{\rm m.o., odd}_{R(d)}(n)|) =(-1)^{(n-1)/d}|{\cal T}^{\rm m.o., good}_{R(d)}(n)|
$$ and $b_{n,k} = (-1)^{(n-k)/d}|{\cal F}^{\rm m.o., good}_{R(d)}(n,k)|$.  
Similarly, if we take $a_n = n!{\bf 1}_{\{n \in R(d)\}}$ then we get 
$$
b_n = (-1)^{(n-1)/d}(|{\cal T}^{\rm l.o., even}_{R(d)}(n)| - |{\cal T}^{\rm l.o., odd}_{R(d)}(n)|) =(-1)^{(n-1)/d}|{\cal T}^{\rm l.o., good}_{R(d)}(n)|
$$ 
and $b_{n,k} = (-1)^{(n-k)/d}|{\cal F}^{\rm l.o., good}_{R(d)}(n,k)|$.
\end{proof}

\begin{proof}[Proof of Corollary \ref{cor:special-cases}]
This easily follows from the definitions and theorems indicated.
\end{proof}

\subsection{Proof of Theorem \ref{thm:Whitney}}

It is clear that $\sigma$ is covered by $\tau$ in $\Pi^{1,d}_n$ if and only if some part of $\tau$ is the union of $d+1$ parts of $\sigma$ and every other part of $\sigma$ is a part of $\tau$.  Thus $\Pi_n^{1,d}$ is ranked, the partitions at rank $k$ are precisely the partitions with $n-kd$ parts, and $W_k(\Pi_n^{1,d}) = {n \brace n- kd}_{\{1, d+1, 2d+1, \ldots\}}$. 

Suppose $\sigma = \{\sigma_1, \ldots, \sigma_k\}$ and $\tau = \{\tau_1, \ldots, \tau_\ell\}$ are partitions in $\Pi_n^{1,d}$.  If $\sigma \leq \tau$, i.e. $\sigma$ is a refinement of $\tau$, 
let $\epsilon = \epsilon(\sigma,\tau) = \{\epsilon_1, \ldots, \epsilon_\ell\}$ be the unique partition of $[k]$ such that for all $1 \leq i \leq \ell$, $\tau_i = \bigcup_{j \in \epsilon_i} \sigma_j$.  We have $\epsilon \in \Pi_k^{1,d}$.  Indeed, since $|\tau_i|  = \sum_{j \in \epsilon_i} |\sigma_j|$,  $|\epsilon_i| \equiv \sum_{j \in \epsilon_i} 1 \equiv \sum_{j \in \epsilon_i} |\sigma_j| \equiv |\tau_i| \equiv 1 \pmod{d}$ for all $i$.

Fix $\sigma \in \Pi_n^{1,d}$ with $k$ parts and let $P =[\sigma, \infty) = \{ \tau \in \Pi_n^{1,d} | \tau \geq \sigma\}$.  We have $\tau \leq \tau'$ in $P$ if and only if $\epsilon(\sigma,\tau) \leq \epsilon(\sigma,\tau')$ in $\Pi_k^{1,d}$.  Thus $[\sigma,\infty)$ is isomorphic to $\Pi_{|\sigma|}^{1,d}$ via the isomorphism $f(\tau) = \epsilon(\sigma,\tau)$. Since the isomorphism type of $[\sigma,\infty)$ depends only on the number of parts of $\sigma$, this type also depends only on the rank of $\sigma$, i.e. $\Pi_n^{1,d}$ is uniform.

We set $W_z(\cdot) = 0$ for non-integer and negative values of $z$ so that $W_{(n-k)/d}(\Pi_n^{1,d}) = {n \brace k}_{\{1, d+1, 2d+1, \ldots\}}$ for all $n,k \geq 1$.  We also set $w_z(\cdot)=0$ for non-integer or negative values of $z$ and show \[ [w_{(n-k)/d}(\Pi_n^{1,d})]_{n,k \geq 1} = \left[{n \brace k}_{\{1 ,d+1, 2d+1, \ldots\}} \right]_{n,k \geq 1}^{-1}.\]  This will prove $w_{(n-k)/d}(\Pi_n^{1,d}) = {n \brace k}^{-1}_{\{1,d+1, 2d+1, \ldots\}}$ or $w_k(\Pi_n^{1,d}) = {n \brace n - kd}^{-1}_{\{1, d+1, 2d+1, \ldots\}}$ as desired.

Let \[S(n,\ell) = \sum_k {n \brace k}_{\{1 , d+1, 2d+1, \ldots\}} w_{(k-\ell)/d}(\Pi_k^{1,d}) = \sum_k W_{(n-k)/d}(\Pi^{1,d}_n) w_{(k-\ell)/d}(\Pi_k^{1,d})\]

We have to show that $S(n,\ell)= {\bf 1}_{\{n=\ell\}}$, for all $n, \ell \geq 1$.  Since the summand in $S(n,\ell)$ is $0$ unless $\ell \leq k \leq n$, we have $S(n,\ell) = 0$ for $\ell > n$.  Clearly if $\ell = n$, $S(n, \ell) = 1$.  We now suppose that $\ell < n$.  The summand is also $0$ unless $k \equiv \ell \pmod{d}$ and $k \equiv n \pmod{d}$.  So $S(n, \ell)=0$ if $n \not \equiv \ell\pmod{d}$. Suppose now that $n \equiv \ell \pmod{d}$.  We may restrict the index of summation to those $k$ for which $k = n - j d$ for some integer $j \geq 0$. (All other terms are $0$.) Fix $j_0$ so that $\ell = n - j_0 d$ and reindex the summation by $j$.  Then \[S(n,\ell) = \sum_{j=0}^{j_0} W_j(\Pi_n^{1,d}) w_{j_0-j}(\Pi_{n-j d}^{1,d}). \]

For $n \geq 1$ let $\zeta = \zeta_n$ and $\mu = \mu_n$ be the zeta and M{\"o}bius functions of $\Pi_n^{1,d}$, i.e. $\zeta(\sigma,\tau) = {\bf 1}_{\{\sigma \leq \tau\}}$ for all $\sigma, \tau \in \Pi_n^{1,d}$ and $[\mu(\sigma,\tau)]_{\sigma,\tau \in \Pi^{1,d}_n} = [\zeta(\sigma,\tau)]_{\sigma,\tau \in \Pi^{1,d}_n}^{-1}$ (see \cite{StanleyEC1}).  Let $\rho_n$ be the rank function of $\Pi_n^{1,d}$, i.e. $\rho_n(\sigma) = (n-k)/d$ where $k$ is the number of parts of $\sigma$.  Let $0_n$ be the unique minimal element of $\Pi_n^{1,d}$, i.e. the partition of $[n]$ into singletons.  

Fix $j$ with $0 \leq j \leq j_0$.  We have \[W_{j}(\Pi_n^{1,d}) = \sum_{\sigma \in \Pi^{1,d}_n, \rho_n(\sigma) = j} \zeta(0_n,\sigma)\] and by definition \[w_{j_0-j}(\Pi_{n-j d}^{1,d}) = \sum_{\epsilon \in \Pi^{1,d}_{n - j d}, \rho_{n - j d}(\epsilon) = j_0 - j} \mu_{n- j d}(0_{n-j d},\epsilon).\]  Fix an element $\sigma$ in $\Pi_n^{1,d}$ with $\rho_n(\sigma) = j$, i.e. with $k = n - j d$ parts.  Then $[\sigma,\infty)$ is isomorphic to $\Pi_{n - j d}^{1,d}$ and so $\mu_{n-j d}(0_{n - j d},\epsilon(\sigma,\tau)) = \mu_n(\sigma,\tau)$ and $\rho_{n- j d}(\epsilon(\sigma,\tau)) = \rho_n(\tau) -\rho_n(\sigma)$ for all $\tau \geq \sigma$.  Thus
\[w_{j_0-j}(\Pi_{n-j d}^{1,d}) = \sum_{\tau \in [\sigma,\infty), \rho_n(\tau) = j_0} \mu_n(\sigma,\tau)\]
and
\[S(n,\ell) = \sum_{j=0}^{j_0} \sum_{\sigma \in \Pi^{1,d}_n, \rho_n(\sigma) = j} \zeta_n(0_n,\sigma) \sum_{\tau \in [\sigma,\infty), \rho_n(\tau) = j_0} \mu_n(\sigma,\tau)\]
\[ = \sum_{\tau \in \Pi_n^{1,d}, \rho_n(\tau) = j_0}  \sum_{\sigma \in [0_n, \tau]} \zeta_n(0_n,\sigma)\mu_n(\sigma,\tau).\]
For each $\tau$ in the summation the inner summation is $0$ as $\mu$ and $\zeta$ are inverses.  

\begin{remark}
This proof was inspired by Exercise 3-130 of \cite{StanleyEC1} (which in turn generalizes Theorem 6 of \cite{Dowling}). The statement in \cite{StanleyEC1} only covers uniform, ranked posets with a $0$ and a $1$, leaving out $\Pi_n^{1,d}$ with $n \not \equiv 1 \pmod{d}$.
\end{remark}

\section{Discussion and some open questions} \label{sec:conclusion}

For $R \subseteq {\mathbb N}$ with $1 \in R$ and with no exposed odds, it is straightforward to enumerate $R$-good and $R(d)$-good increasingly ordered, min-first ordered and linearly ordered trees by number of leaves. Indeed, from our results we have the following for all such $R$; here we use the notation $[x^n/n!]f(x)$ to denote the coefficient of $x^n/n!$ in the Taylor series of $f(x)$, and recall that $f^{-1}(x)$ denotes the compositional inverse or series reversion of $f(x)$.  
\begin{itemize}
\item The number of $R$-good increasingly ordered trees with $n$ leaves is $(-1)^{n-1}[x^n/n!]f_1^{-1}(x)$ where $f_1(x)=\sum_{k \in R} x^k/k!$, and 
\item the number of $R(d)$-good increasingly ordered trees with $d(n-1)+1$ leaves is 
$$
(-1)^{n-1}\left[\frac{x^{d(n-1)+1}}{(d(n-1)+1)!}\right]f_2^{-1}(x)
$$ 
where $f_2(x)=\sum_{k \in R} x^{d(k-1)+1}/(d(k-1)+1)!$.  
\end{itemize}
The same holds for min-first ordered trees, with 
$$
f_1(x)=\sum_{k \in R} \frac{x^k}{k}, \quad f_2(x)=\sum_{k \in R} \frac{x^{d(k-1)+1}}{d(k-1)+1},
$$
and for linearly ordered trees, with 
$$f_1(x)=\sum_{k \in R} x^k, \quad f_2(x)=\sum_{k \in R} x^{d(k-1)+1}.$$

For example, the series reversion of $f(x)=x+x^2/2$ is $$f^{-1}(x)=\sum_{n \geq 1} (-1)^{n-1}\frac{(2n-3)!!x^n}{n!}$$ (where $m!!=m(m-2)(m-4)\ldots $ is the double factorial), and so the sequence of both $[2]$-good increasingly ordered trees and $[2]$-good min-first ordered trees is 
$(1, 1, 3, 15, 105, 945, 10395,\ldots)$ \cite[A001147]{sloan}, while the series reversion of $g(x)=x+x^2$ is $$g^{-1}(x)=\sum_{n \geq 1} (-1)^{n-1}\frac{(2n-2)!x^n}{(n-1)!n!},$$ and so the sequence of $[2]$-good linearly ordered trees is 
$(1, 2, 12, 120, 1680, 30240, 665280, \ldots)$ \cite[A001813]{sloan}.

Another interesting example relates to the following special functions. For $d \geq 1$ the {\em hyperbolic function of order $d$ of the first kind} (see for example \cite{Ungar}) is the function $H_{d,1}(x)$ defined by the power series 
$$
H_{d,1}(x) = \sum_{n \geq 1} \frac{x^{d(n-1)+1}}{(d(n-1)+1)!}; 
$$
so for example $H_{1,1}(x)=e^x-1$ and $H_{2,1}(x)=\sinh x$. The study of these functions goes back to the mid-1700's. As an immediate by-product of Theorem \ref{thm:single-forest-stretched} and Theorem \ref{thm:Whitney} we obtain combinatorial interpretations for the coefficients of the compositional inverses of these functions and their connection to Whitney numbers of the poset $\Pi_{d(n-1)+1}^{1,d}$. 
\begin{cor} \label{cor-hyperbolic-AAA}
For $d \geq 1$, let $h_{d,1}(x)$ be the compositional inverse of $H_{d,1}(x)$ (satisfying $h_{d,1}(H_{d,1}(x))) = H_{d,1}(h_{d,1}(x)))=x$ for all $x$). Then writing $h_{d,1}(x)$ in the form
$$
h_{d,1}(x) = \sum_{n \geq 1} (-1)^{n-1}h_n \frac{x^{d(n-1)+1}}{(d(n-1)+1)!}
$$
we have  
\begin{enumerate}
	\item[(a)]$h_n$ is the number of increasingly ordered trees with $d(n-1)+1$ leaves that are ${\mathbb N}(d)$-good, i.e. have all vertices of degree $d+1$ or $0$ and all left-most children of degree $0$; and
    \item[(b)] $w_{n-1}(\Pi_{d(n-1)+1}^{1,d}) = (-1)^{n-1} h_n$,  i.e., the Whitney numbers of the first kind of the poset $\Pi_{d(n-1)+1}^{1,d}$ are the coefficients of the exponential generating function of the compositional inverse of $H_{d,1}(x)$.
\end{enumerate}
\end{cor}
As discussed after Definition \ref{def:R-good}, there are $(n-1)!$ ${\mathbb N}(1)$-good increasingly ordered trees with $n$ leaves, and indeed the compositional inverse of $H_{1,1}(x)=e^x-1$ is $\log(1+x) =\sum_{n \geq 1} (-1)^{n-1} (n-1)!x^n/n!$.

For $d=2$ the sequence $n$th term is the number of ${\mathbb N}(2)$-good increasingly ordered trees with $2(n-1)+1$ leaves begins $(1, 1, 9, 225, 11025, 893025, 108056025,\ldots)$, and is the sequence of squares of double factorials of odd numbers \cite[A001818]{sloan}. It is well-known that this sequence arises in the power series of the inverse of the hyperbolic sine function. For $d=3$ it begins $(1,1,34,5446, 2405116, 2261938588, 3887833883752, \ldots)$; this sequence does not appear in \cite{sloan}. 

\medskip

We have given combinatorial interpretations for each each of ${n \brace k}^{-1}_{R}$, ${n \brack k}^{-1}_{R}$ and $L(n,k)^{-1}_{R}$ for all ${R}$ with $1 \in {R}$, but for many ${R}$ these interpretations are as the difference in cardinalities of two sets of forests. Only for $R$ and $R(d)$ with $1 \in R$ and with no exposed odds can we interpret the inverse entries as counts of single sets of forests. In all of these special cases we have the crucial property that the compositional inverses of $\sum_{n \in R} x^n/n!$, $\sum_{n \in R} x^n/n$, $\sum_{n \in R} x^n$, $\sum_{n \in R} x^{d(n-1)+1}/(d(n-1)+1)!$, $\sum_{n \in R} x^{d(n-1)+1}/(d(n-1)+1)$ and $\sum_{n \in R} x^{d(n-1)+1}$ each have alternating coefficient sequences (in the latter three cases, alternating along an arithmetic progression). Here we say that a series $\sum_{n \geq 1} c_n x^n$ with $c_1 > 0$ is \emph{alternating} if $(-1)^{n-1}c_n \geq 0$ for all $n \geq 1$; it is alternating along the arithmetic progression $A=\{1,d+1,2d+1,\ldots\}$ if $c_n = 0$ for all $n \not \in A$ and if $(-1)^{k}c_{kd+1} \geq 0$ for all $k \geq 0$.

This raises a number of natural questions. 

\begin{question} \label{quest:characterize}
Can we characterize those $R \subseteq {\mathbb N}$ with $1 \in R$ for which the compositional inverse of $\sum_{n \in R} x^n/n!$ ($\sum_{n \in R} x^n/n$, $\sum_{n \in R} x^n$) has an alternating coefficient sequence or one alternating along an arithmetic progression starting at $1$?
\end{question}

\begin{question} \label{quest:combinatorics}
For those $R$, is there an analog of Algorithm \ref{alg:algorithm} that furnishes a combinatorial interpretation of the numbers ${n \brace k}^{-1}_{R}$, etc.?
\end{question}

In the case of $\sum_{n \in R} x^n$, we can say definitively that the characterization sought in Question \ref{quest:characterize} is not simply having no exposed odds. Let $f(x)$ be a power series with $\ord(f(x))=1$ and with a positive coefficient of $x$. In what follows we say that a series $\sum_{n \geq 0} c_n x^n$ with $c_0 > 0$ is \emph{alternating} if $(-1)^nc_n \geq 0$ for all $n \geq 0$.
\begin{claim} \label{clm:suff}
A sufficient condition for the compositional inverse $f^{-1}(x)$ of $f(x)$ to be alternating is that $x/f(x)$ is alternating.
\end{claim}
\begin{proof}
Since the product of alternating power series with positive constant terms is again alternating with positive constant term, under the hypothesis of the claim we get that for all $n \geq 1$ the power series of $(x/f(x))^n$ is alternating. The Lagrange inversion formula (see e.g. \cite[Chapter 5]{StanleyEC2}), which says that for all $n$ the coefficient of $x^n$ in $f^{-1}(x)$ is the same as $(1/n)$ times the coefficient of $x^{n-1}$ in $(x/f(x))^n$, then says that the sign of the coefficient of $x^n$ in $f^{-1}(x)$ is $(-1)^{n-1}$ or $0$.
\end{proof}

This is not a terribly useful test for the power series that come up when studying restricted Stirling numbers, but it is quite useful for restricted Lah numbers, where the series under consideration take the form $f(x)=\sum_{n \in R} x^n$, and the geometric series can sometimes be used to find an explicit expression for the coefficients of the power series of $x/f(x)$. For example, when $R=\{1,2,r+1,r+2\}$ for $r \geq 2$, we have 
\begin{eqnarray*}
\frac{x}{x+x^2+x^{r+1}+x^{r+2}} & = & \frac{1}{(1+x)(1+x^r)} \\
& = & \left\{ \begin{array}{ll}
\sum_{k=1}^{\infty} (-1)^{k-1}k \sum_{j=0}^{r-1} (-1)^{j}x^{(k-1)r+j} & \mbox{if $r$ odd} \\
\sum_{k=1}^{\infty} \sum_{j=0}^{r-1} (-1)^{j}x^{2(k-1)r+j} & \mbox{if $r$ even},
\end{array}
\right.
\end{eqnarray*}
which is alternating. This shows that $L(n,k)_R^{-1}$ has sign $(-1)^{n-k}$ (or $0$) for all $n, k \geq 1$, whenever $R$ is of the form $\{1,2,r+1,r+2\}$ for $r \geq 2$; but only in the case $r=2$ is this a set $R$ with $1 \in R$ and with no exposed odds.

There is some computational evidence in favor of an affirmative answer to the following question, but perhaps not enough to merit forming a conjecture.

\begin{question} \label{quest:three-musketeers}
Is it the case that for $R \subseteq {\mathbb N}$ with $1 \in R$, we have that the inverse of $\sum_{n \in R} x^n/n!$ is alternating if and only if  the inverse of $\sum_{n \in R} x^n/n$ is alternating and if and only if the inverse of $\sum_{n \in R} x^n$ is alternating?
\end{question}

In light of the discussion after Question \ref{quest:combinatorics}, it is worth noting that the compositional inverses of both $x+x^2/2+x^4/24+x^5/120$ and $x+x^2/2+x^4/4+x^5/5$ are alternating for their first 1200 terms. 

We have shown in this paper, by a {\em combinatorial} argument (Algorithm \ref{alg:algorithm}) that if $R \subseteq {\mathbb N}$ with $1 \in R$ has no exposed odds, then $f(x) = \sum_{n \in R} x^n/n!$, $g(x) = \sum_{n \in R} x^n/n$ and $h(x) = \sum_{n \in R} x^n$ have compositional inverses with alternating coefficient sequences.  In \cite[Section 5]{EGS-draft} we also show $h(x) = \sum_{n \in R} x^n$ has an alternating inverse by a different combinatorial argument expressing inverse Lah numbers in terms of Dyck paths.  There we also showed analytically that $x/h(x)$ is alternating.  Together with Claim \ref{clm:suff} this gives an analytical proof that $h^{-1}(x)$ is alternating.  

This leads us to the following non-combinatorial question: are there analytical proofs that $f^{-1}(x)$ and $g^{-1}(x)$ are alternating?  We do not even know of an analytical way of showing, for example, that $x+x^2/2+x^3/3+x^4/4$, the degree four Taylor approximation to $\log(1+x)$, has alternating compositional inverse (note that $x/(x+x^2/2+x^3/3+x^4/4)$ does not have an alternating power series, so we cannot apply Claim \ref{clm:suff}).

\section{Acknowledgements}

The second author thanks Hannah Porter and David Radnell for helpful discussions. The authors thank Patricia Hersh for helpful discussions.


\begin{thebibliography}{99}

\bibitem{Barry} P. Barry, {\em Riordan Arrays: A Primer},  Logic Press, Naas, Ireland, 2016.

\bibitem{BelbachirBousbaa}
H. Belbachir and I. E. Bousbaa, Associated Lah numbers and $r$-Stirling numbers, arXiv:1404.5573 (2014).

\bibitem {Calderbank} A.R. Calderbank, P. Hanlon, and R.W. Robinson, Partitions into even and odd block size and some unusual characters of the symmetric groups, {\em Proc. London Math. Soc. (3)} {\bf 53} (1986), 288--320.

\bibitem{ChoiLongNgSmith}
J. Y. Choi, L. Long, S.-H. Ng and J. Smith, Reciprocity for multirestricted Stirling numbers, {\em J. Combin. Theory Ser. A} {\bf 113} (2006), 1050--1060.

\bibitem{ChoiSmith}
J. Y. Choi and J. D. H. Smith, On the combinatorics of multi-restricted numbers, {\em Ars Combin.} {\bf 75} (2005), 45--63.

\bibitem{ChoiSmith2}
J. Y. Choi and J. D. H. Smith, On the unimodality and combinatorics of Bessel numbers, {\em Discrete Math.} {\bf 264} (2003) 45--53.

\bibitem{Comtet}
L. Comtet, {\em Advanced Combinatorics}, Reidel, Boston, 1974.

\bibitem{Drake} B. Drake, An inversion theorem for labeled trees and some limits of areas under lattice paths, Ph.D. Thesis, Brandeis University, 2008.

\bibitem{Dowling} T. A. Dowling, A class of geometric lattices based on finite groups, {\em J. Combin. Theory Ser. B} {\bf 14} (1973), 61--86.

\bibitem{EGS-draft}
J. Engbers, D. Galvin and C. Smyth, Restricted Stirling and Lah numbers and their inverses, arXiv:1610.05803v1 (2016).

\bibitem{Ginzburg-Kapranov} V. Ginzburg and M. Kapranov, Koszul duality for operads, {\em Duke Math. J} {\bf 76} (1994), 203--272.

\bibitem{MacKay-et-al} J.H. MacKay, J. Towber, S. S.-S. Wang, and D. Wright, Reversion of a system of power series with application to combinatorics, Preprint 1992.

\bibitem{sloan}
N. Sloane, editor, The on-line encyclopedia of integer sequences, \url{oeis.org}.   

\bibitem{StanleyEC1}
R. Stanley, {\em Enumerative Combinatorics, vol. 1, 2nd ed.}, Cambridge University Press, Cambridge, 2011.

\bibitem{StanleyEC2}
R. Stanley, {\em Enumerative Combinatorics, vol. 2}, Cambridge University Press, Cambridge, 1999.

\bibitem{StanleyExp} R. Stanley, Exponential structures, {\em Stud. Appl. Math.} {\bf 59} (1978), 73--82.

\bibitem{Sylvester} G. Sylvester, Continuous spin Ising ferromagnets, PhD Thesis, Massachusetts Institute of Technology, 1976.

\bibitem{Taylor} J. Taylor, Formal group laws and hypergraph colorings, Ph.D. thesis, University of Washington, 2016.

\bibitem{Ungar}
A. Ungar, Generalized Hyperbolic Functions, {\em Amer. Math. Monthly}  {\bf 89} (1982), 688--691.

\bibitem{Wachs}

M. Wachs, Whitney homology of semipure shellable posets, {\em J. Algebraic Combin.} {\bf 9} (1999), 173--207.

\bibitem{Wilf}
H. Wilf, {\em generatingfunctionology, 3rd ed.}, A. K. Peters, Wellesley, 2006.

\bibitem{Wright}
D. Wright, The tree formulas for reversion of power series, {\em J. Pure Appl. Algebra} {\bf 57} (1989), 191--211.

\end{thebibliography}
\end{document}